\documentclass[reqno,11pt]{amsart}
\usepackage{setspace,tikz,xcolor,mathrsfs,listings,multicol,amssymb}
\usepackage[vcentermath]{youngtab}
\usepackage[margin=1in,includefoot,footskip=30pt]{geometry}
\usepackage{rotating,enumerate,booktabs}
\usepackage[centertableaux]{ytableau}
\usepackage[all,cmtip]{xy}
\usetikzlibrary{arrows,matrix}
\tikzset{tab/.style={matrix of math nodes,column sep=-.35, row sep=-.35,text height=7pt,text width=7pt,align=center,inner sep=2,font=\footnotesize}}
\usepackage[colorlinks=true, pdfstartview=FitV, linkcolor=black, citecolor=black, urlcolor=black]{hyperref}

\usepackage{tikz}\usetikzlibrary{graphs, quotes, fit, positioning, matrix, calc, decorations.markings, angles, decorations.pathmorphing, decorations.pathreplacing}
\tikzset{math mode/.style = {execute at begin node=$, execute at end node=$}}
\tikzset{arrow/.style={postaction={decorate,thick,decoration={markings,mark = at position #1 with {\arrow{>}}}}},arrow/.default=0.5}
\tikzset{invarrow/.style={postaction={decorate,thick,decoration={markings,mark = at position #1 with {\arrow{<}}}}},invarrow/.default=0.5}
\newcommand\bbull[3]{\filldraw[fill=black!35!blue, draw=black] (#1,#2) circle (#3cm);}
\newcommand\ebull[3]{\filldraw[fill=white,draw=black] (#1,#2) circle (#3cm);}
\renewcommand\ss{\scriptstyle}
\tikzset{blob/.style={circle,thin,draw=black,fill=#1,inner sep=1.8pt}}

\newif\iftikz
\tikztrue

\newcommand{\g}{\mathfrak{g}}
\newcommand{\fsl}{\mathfrak{sl}}

\newcommand{\eschur}{E} 
\newcommand{\dualfact}{\mathcal{E}} 
\newcommand{\fw}{\Lambda} 
\newcommand{\coroot}{\alpha^{\vee}} 
\newcommand{\inner}[2]{\langle #1, #2 \rangle}
\newcommand{\iso}{\cong}

\newcommand{\ssyt}{\operatorname{SSYT}}

\newcommand{\elt}{\operatorname{ELT}}

\newcommand{\Gr}{\operatorname{Gr}}  
\newcommand{\GL}{\operatorname{GL}}  

\newcommand{\elb}[2]{\raisebox{-4pt}{$\substack{\displaystyle #1 \\[1pt] \displaystyle #2}$}} 
\newcommand{\elbb}[2]{\raisebox{5pt}{$\substack{\displaystyle #1 \\[1pt] \displaystyle #2}$}} 
\newcommand{\abs}[1]{\lvert #1 \rvert}

\newcommand{\ket}[1]{| #1 \rangle}
\newcommand{\bra}[1]{\left\langle #1\right|}
\newcommand{\braket}[2]{\left\langle #1|#2\right\rangle}
\newcommand{\normord}[1]{: \mathrel{#1} :}  
\newcommand{\bbb}{\mathsf{b}}
\newcommand{\cp}{{\color{blue}+}}
\newcommand{\cm}{{\color{darkred}-}}
\newcommand{\zero}{\mathbf{0}}

\DeclareMathOperator{\word}{rwd} 
\DeclareMathOperator{\wt}{wt} 
\DeclareMathOperator{\Bwt}{Bwt} 
\DeclareMathOperator{\ch}{ch} 
\DeclareMathOperator{\sig}{sig} 
\DeclareMathOperator{\End}{End} 

\newcommand{\mcB}{\mathcal{B}}
\newcommand{\mcC}{\mathcal{C}}

\newcommand{\mcI}{\mathcal{I}}
\newcommand{\mcS}{\mathcal{S}}
\newcommand{\mcV}{\mathcal{V}}

\newcommand{\TMat}{\mathbf{T}}
\newcommand{\tmat}{\mathbf{t}}

\newcommand{\fE}{\mathfrak{E}}

\newcommand{\fM}{\mathfrak{M}}

\newcommand{\sft}{\mathsf{t}}

\newcommand{\xx}{\mathbf{x}}
\newcommand{\yy}{\mathbf{y}}

\newcommand{\aaa}{\mathbf{a}}

\newcommand{\ZZ}{\mathbb{Z}}

\newcommand{\CC}{\mathbb{C}}

\definecolor{darkred}{rgb}{0.7,0,0} 
\newcommand{\defn}[1]{{\color{darkred}\emph{#1}}} 

\definecolor{UQgold}{RGB}{196, 158, 54} 
\definecolor{UQpurple}{RGB}{73, 7, 94} 
\definecolor{OCUenji}{RGB}{153,0,51} 
\definecolor{OCUsapphire}{RGB}{0,51,102} 

\tikzset{scol/.style={red}}
\tikzset{ddscol/.style={cyan}}
\tikzset{dscol/.style={blue}}
\tikzset{mscol/.style={OCUsapphire}}
\usepackage{listings}
\lstdefinelanguage{Sage}[]{Python}
{morekeywords={False,sage,True},sensitive=true}
\definecolor{dblackcolor}{rgb}{0.0,0.0,0.0}
\definecolor{dbluecolor}{rgb}{0.01,0.02,0.7}
\definecolor{dgreencolor}{rgb}{0.2,0.4,0.0}
\definecolor{dgraycolor}{rgb}{0.30,0.3,0.30}
\definecolor{pinkish}{rgb}{1.0, 0.44, 0.37}

\theoremstyle{plain}
\newtheorem{thm}{Theorem}[section]
\newtheorem{lemma}[thm]{Lemma}
\newtheorem{conj}[thm]{Conjecture}
\newtheorem{prop}[thm]{Proposition}
\newtheorem{cor}[thm]{Corollary}
\theoremstyle{definition}
\newtheorem{dfn}[thm]{Definition}
\newtheorem{ex}[thm]{Example}
\newtheorem{remark}[thm]{Remark}
\newtheorem{problem}[thm]{Problem}
\numberwithin{equation}{section}

\usepackage[colorinlistoftodos]{todonotes}

\begin{document}
\title[Integrable crystals for ELT]{Integrable systems and crystals for edge labeled tableaux}

\author[A.~Gunna]{Ajeeth Gunna}
\address[A.~Gunna]{School of Mathematics and Statistics, University of Melbourne, Parkville, Victoria 3010, Australia}
\email{agunna@student.unimelb.edu.au}

\author[T.~Scrimshaw]{Travis Scrimshaw}
\address[T.~Scrimshaw]{OCAMI, Osaka City University, 3--3--138 Sugimoto, Sumiyoshi-ku, Osaka 558-8585, Japan}
\email{tcscrims@gmail.com}
\urladdr{https://tscrim.github.io}


\thanks{
T.S.~was partially supported by Grant-in-Aid for JSPS Fellows 21F51028.
}

\begin{abstract}
We introduce the edge Schur functions $E^{\lambda}$ that are defined as a generating series over edge labeled tableaux.
We formulate $E^{\lambda}$ as the partition function for a solvable lattice model, which we use to show they are symmetric polynomials and derive a Cauchy-type identity with factorial Schur polynomials.
Finally, we give a crystal structure on edge labeled tableau to give a positive Schur polynomial expansion of $E^{\lambda}$ and show it intertwines with an uncrowding algorithm.
\end{abstract}

\maketitle

\section{Introduction}
\label{sec:introduction}

A classical and important object from algebraic geometry is the Grassmannian $X = \Gr(k,n)$, the set of $k$-dimensional planes in $\CC^n$, which is a projective variety from the Pl\"ucker embedding.
We can realize $\Gr(k,n) = \GL_n / P_k$, where $\GL_n$ is the general linear group (over $\CC$) and $P_k$ is the maximal parabolic subgroup consisting of (invertible) block upper triangular matrices with block sizes $(k, n-k)$.
This gives $\Gr(k, n)$ the structure of a complex manifold.
To construct a cell decomposition of the Grassmannian, we first need to define $B$ as the subgroup of upper triangular matrices and identify permutations with their permutation matrix in $\GL_n$.
The decomposition we consider is indexed by partitions $\lambda$ contained inside of a $k \times (n-k)$ rectangle, and for each such $\lambda$, we define a permutation
\[
w_{\lambda} = [\lambda_k+1, \dotsc, \lambda_1+k, k+1-\lambda'_1, \dotsc, n-\lambda'_{n-k}]
\]
written in one-line notation, where $\lambda'$ is the conjugate partition.
Then, the $B$-orbits $B w_{\lambda}P_k / P_k$ are algebraic affine spaces of (complex) dimension $\abs{\lambda}$ and give the cell decomposition of $\Gr(k,n)$~\cite{Ehresmann34} essentially as the $LU$ decomposition of a matrix.
By taking the closure under the Zariski topology of the $B$-orbits, we obtain the Schubert varieties $X_{\lambda} = \overline{B w_{\lambda}P_k/P_k}$, which are closed irreducible varieties giving the closed cells with $X_{\mu} \subseteq X_{\lambda}$ if and only if $\mu \subseteq \lambda$.
We could also use the subgroup $B^-$ of lower triangular matrices to form the opposite Schubert varieties $X^{\lambda} = \overline{B^- w_{\lambda} P_k/P_k}$ with (complex) codimension $\abs{\lambda}$.

The singular cohomology ring $H^{\bullet}(X)$ of the Grassmannian is one of the fundamental topological invariants.
For any closed subvariety $Y \subseteq X$ of dimension $k$, we can construct a cohomology class $[Y] \in H^{2k}(X)$.
We have a basis for $H^{\bullet}(M)$ corresponding to Schubert varieties $\{[X_{\lambda}]\}_{\lambda}$ (from~\cite[\S8]{vanderWaerden30} interpreted using later results) as they are closed irreducible subvarieties from the cell decomposition.
The cup product corresponds to the intersection of Schubert varieties under Poincar\'e duality and gives finer information than the homology $H_{\bullet}(X)$.
The structure coefficients $[X_{\lambda}] \smile [X_{\mu}] = \sum_{\nu} c_{\lambda\mu}^{\nu} [X_{\nu}]$ correspond to the decomposition of the intersection $X_{\lambda} \cap X_{\mu}$ into irreducibles $X_{\nu}$ (taking representatives that intersect transversally).
From this geometric interpretation, $c_{\lambda\mu}^{\nu}$ is a positive integer that can be described combinatorially.
We refer the reader to~\cite{Fulton} for more details.

From the classical works of Giambelli~\cite{Giambelli02} and Pieri~\cite{Pieri93}, we can represent the cohomology class of Schubert varieties by Schur functions $[X_{\lambda}] \leftrightarrow s_{\lambda}(\xx_k)$ in $k$ variables $\xx_k = (x_1, \dotsc, x_k)$.
Furthermore, we have a simple description of $H^{\bullet}(X) \iso \Lambda / \mcI_k$~\cite{Borel53}, where $\Lambda$ is the ring of symmetric functions and $\mcI_k = (s_{\lambda}(\xx_k) \mid \lambda \not\subseteq (n-k)^k)$ is the ideal generated by Schur functions such that $\lambda$ is not contained inside of the $k \times (n-k)$ rectangle.
Since this is an isomorphism of rings (or $\CC$-algebras), we have that the Littlewood--Richardson coefficients $c_{\lambda\mu}^{\nu}$ are nonnegative integers, which have a well-known combinatorial interpretation (see, \textit{e.g.},~\cite{ECII}).
We can then take the natural embeddings $\iota_n \colon \GL_n / P_k \to \GL_{n+1} / P_k$ to build the (direct) limit $n \to \infty$ object (which turns out to be the classifying space of the unitary group $\mathrm{U}_n$) and the cohomology rings are symmetric functions in $k$ variables with natural quotients between them (and hence $H^{\bullet}(\Gr(k,\infty))$ is the inverse limit).

We can build the bialgebra structure on symmetric functions by using the natural embedding
\begin{equation}
\label{eq:grassmannian_embedding}
\Gr(j,m) \times \Gr(k,n) \to \Gr(j+k,m+n),
\qquad\qquad
(V, W) \mapsto V \oplus W.
\end{equation}
By the K\"unneth formula, we have an induced (surjective) map on cohomology
\[
H^a(\Gr(j+k,m+n)) \to H^a(\Gr(j,m) \times \Gr(k,n)) \iso \bigoplus_{a=b+c} H^b(\Gr(j,m)) \otimes H^c(\Gr(k,n))
\]
that corresponds to the coproduct on symmetric functions, reflecting the idea that the coproduct is splitting the input variables.
By instead considering homology, this becomes a product on $H_{\bullet}(X)$, which makes it isomorphic (as an algebra) to symmetric functions (see, \textit{e.g.},~\cite[Sec.~1.1]{KL15}), where the Schur functions again arise as the classes for Schubert varieties.
(The coproduct on homology is again from the K\"unneth formula, but this time applied to the diagonal embedding.)
The geometric analog of the antipode is the natural duality of the Grassmannian $\Gr(k, n) \iso \Gr(n-k,n)$ sending $w_{\lambda} \mapsto w_0 w_{\lambda} w_0 = w_{\lambda'}$.
Poincar\'e duality identifies the homology with the cohomology, encoding the fact that symmetric functions are a self-dual Hopf algebra, but we need to use opposite Schubert varities on one side in order to make a direct comparison.
Furthermore, for $\abs{\lambda} = \abs{\mu}$, we have $\abs{X^{\lambda} \cap X_{\mu}} = \delta_{\lambda\mu}$ (see, \textit{e.g.},~\cite[Sec.~9.4]{Fulton}) being equivalent to the Schur functions being an orthonormal basis under the pairing, which is the Hall inner product.
See~\cite[Sec.~9.5]{LamPyl07} and~\cite{TY10} for these statements in $K$-theory.

One subtle point about the product on the homology above is that we are changing the space at each step, and so the product only makes sense in the corresponding direct sum of homologies (see~\cite[Sec.~1.1]{KL15} for a precise statement).
The geometric interpretation of this is to instead use the infinite union of Grassmannians $\Gr = \bigcup_{k,n=0}^{\infty} \Gr(k, n)$, but it is sufficient to consider $\bigcup_{k=0}^{\infty} \Gr(k,2k)$.
This is the infinite Grassmannian, which can be described as the virtual dimension zero part of the Sato Grassmannian and is homotopy equivalent to the loop space $\Omega \mathrm{SU}_{\infty}$.
We also encounter this if we try to take the limit $k \to \infty$ of $\Gr(k, \infty)$ as we have to make sense (and be more careful) of what an infinite dimensional subspace of an infinite dimensional space is.
In taking these limits, we are naturally lead to back-stable Schubert calculus studied in detail in~\cite{LLS21}, where the isomorphism with symmetric functions (in infinitely many variables) was given in~\cite[Thm.~6.7]{LLS21}.
In particular, this could be seen by taking the limit $\Gr(k, \infty)$ as $k \to \infty$ but reindexing the $\CC^{\infty}$ basis to be $\{\epsilon_{-k+1}, \ldots, \epsilon_{-1}, \epsilon_0, \epsilon_1, \ldots\}$ for each $k$~\cite[Prop.~3.4]{LLS21}.

Now let us consider the maximal torus $T$ of invertible diagonal matrices and look at the $T$-equivariant cohomology ring $H_T^{\bullet}(X)$.
By noting that each $B$-orbit contains a unique $T$-fixed point $w_{\lambda} P_k/P_k$, we can again use the Schubert varieties to construct a basis of the $T$-equivariant cohomology ring $H_T^{\bullet}(X)$ by $T$-equivariant cohomology classes $[X_{\lambda}]_T$.
We can represent $[X_{\lambda}]_T \in H_T^{\bullet}(X)$ by the factorial Schur functions $s_{\lambda}(\xx_n|\aaa)$~\cite{LS82II,Knutsontaopuzzles} with equivariant parameters $\aaa = (\ldots, a_{-1}, a_0, a_1, \ldots)$, although the precise quotient isomorphic to $H_T^{\bullet}(X)$ is more subtle.
Despite this, a natural question is to compute the structure (or Littlewood--Richardson) coefficients $c_{\lambda\mu}^{\nu}(\aaa)$ for the factorial Schur functions to determine the structure coefficients of $H_T^{\bullet}(X)$.

A combinatorial rule for $c_{\lambda\mu}^{\nu}(\aaa)$ was first accomplished by Molev and Sagan~\cite{Molevsagan1997ALR} in the context of a more general problem.
While these are not manifestly positive, Graham~\cite{Graham1999PositivityIE} rephrased it to be a positive sum of products of certain binomials that correspond to the positive simple roots for $\GL_n$.
In this sense, the Molev--Sagan description is positive as noted by Zinn-Justin~\cite{ZJ09}.
Another positive (in the above sense) formula was given in terms of tilings by Knutson and Tao~\cite{Knutsontaopuzzles}, which was later shown to be governed by a quantum integrable system by Zinn-Justin~\cite{ZJ09}.
However, one might ask for a (skew) tableau rule like one of the classical Littlewood--Richardson rule.
Such a formulation was given by Thomas and Yong~\cite{TY18} by introducing edge labeled tableaux.
The $k \to \infty$ limit for the $T$-equivariant cohomology ring still works using the infinite Grassmannian~\cite[Thm.~6.7]{LLS21}.

Now if we want to examine a potential Hopf algebra structure for $H_T^{\bullet}(X)$, we would need the $T$ action to respect the embedding~\eqref{eq:grassmannian_embedding}.
Unfortunately, in order to deal with the reindexing, we have to coalesce the torus $T$ down to a single circle action $\CC^{\times}$, which means almost setting the involved equivariant parameters $\aaa$ in the limit to a single constant $\alpha$.
More precisely, we take $a_i = \alpha$ for $i < 0$ and $a_i = 0$ for $i \geq 0$.
This ring was studied by Knutson and Lederer~\cite{KL15}, where they gave a combinatorial rule for the corresponding Littlewood--Richardson coefficients in terms of pipe dreams.
By taking the Hopf dual of $H_T^{\bullet}(\Gr)$ (with the Hopf structure coming algebraically from Molev~\cite{Molev-dualschur}), we have that the dual basis $\widehat{s}_{\lambda}(\xx|\aaa)$ of the factorial Schur functions introduced by Molev~\cite{Molev-dualschur} are representatives for the Schubert varieties in the equivariant homology $H_{\bullet}^T(\Gr)$~\cite[Prop.~8.1]{LLS21}.
We remark that this dual basis is not under the Hall inner product but instead under a simple modification~\cite[Eq.~(5.4)]{Molev-dualschur} that interacts well with localization.
Thus, using the aforementioned specialization of $\aaa$, we arrive at an isomorphic ring and basis~\cite[Thm.~8.12]{LLS21} for the $\CC^{\times}$-equivariant cohomology studied by Knutson and Lederer~\cite{KL15}.

Armed with our geometric and algebraic formulations, we want to explore the branching rule by pulling back Schubert classes along the natural embedding $\iota_n$.
In the nonequivariant case, we have the skew Schur functions
\[
s_{\lambda}(\xx, \yy) = \sum_{\mu \subseteq \lambda} s_{\lambda/\mu}(\yy) s_{\mu}(\xx),
\]
which can also be described by the skew operation, which is the adjoint action to multiplication and geometrically corresponds to the cap product.
Therefore, we have
\[
\inner{s_{\mu}^{\perp}s_{\lambda}}{s_{\nu}} = \inner{s_{\lambda}}{s_{\mu}s_{\nu}},
\qquad\qquad
s_{\lambda/\mu}(\yy) = s_{\mu}^{\perp} s_{\lambda}(\yy) = \sum_{\nu} c_{\mu\nu}^{\lambda} s_{\nu}(\yy),
\]
which is reflected in a Littlewood--Richardson rule that uses skew tableau.
Since our multiplication rule uses edge-labeled tableaux, we should be able to define the dual basis under the Hall inner product up to a simple factor based on the Cauchy identity from~\cite[Eq.~(1.4)]{Molev-dualschur}.
Furthermore, edge labeled tableaux were used to prove the equivariant saturation theorem by Anderson, Richmond, and Yong~\cite{ARY} and are expected to be natural objects for the orthogonal and isotropic Grassmannian~\cite{RYY19}.

Molev's dual Schur functions do not have a nice branching rule formulation and the tableau weights are described by rational functions.
Motivated by the above and finding a modification of Molev's dual functions that are closer to Schur functions, we introduce the generating function $\eschur^{\lambda}(\xx|\aaa)$ of edge labeled tableaux of shape $\lambda$, which we call the edge Schur functions.
We use a different weight that is reminiscent of the weight used in computing (refined) symmetric Grothendieck polynomials~\cite{CP21}.
We first construct a solvable lattice model whose partition function is $E^{\lambda}(\xx|\aaa)$ and use the integrability to show (Theorem~\ref{thm:elambda-symmetric}) that $E^{\lambda}(\xx|\aaa)$ is a symmetric function in the $\xx$ variables (in an appropriate completion).
Our lattice model is based on the branching rule and the simple description of the skew edge Schur function in a single variable.
It is a simple consequence that the edge label functions form a basis since $E^{\lambda}(\xx|\aaa) = s_{\lambda}(\xx) + HOT$, where $HOT$ are higher degree terms.

Next, we derive a skew Cauchy-like identity with factorial Schur functions (Theorem~\ref{thm:cauchy}) by showing the corresponding transfer matrices commute up to a scalar.
Here we use the lattice model for $s_{\lambda}(\xx|\aaa)$ given by Zinn-Justin~\cite{ZJ09}.
Consequently, we could derive edge labeled tableau by instead first constructing a lattice model with these properties and then building the combinatorics.
In Section~\ref{sec:edge_variations}, we then discuss some variations and their relationship with $\widehat{s}_{\lambda}(\xx|\aaa)$.
In particular, we see that $\widehat{s}_{\lambda/\mu}(\xx|\alpha) = s_{\lambda/\mu}[\xx / (1 - \alpha\xx)]$, which was recently shown to be connected with quantum transportation~\cite{BN21}.
This is also a specialization of the canonical Grothendieck polynomials introduced by Yeliussizov~\cite{Yel17}, which loosely explains the resemblance with K-theoretic Schubert calculus.
However, this seems to be different than the $K$-theory analogs for the direct sum map~\eqref{eq:grassmannian_embedding} studied in~\cite{TY10}.

Finally, a natural question would be to compute the Schur expansion of the edge Schur functions.
We do so by constructing the crystal structure on edge labeled tableau that could be considered as the analog of~\cite{MPS21,HS20} where the atomic object are the diagonals.
Indeed, we explicitly biject each diagonal with a hook shape and then use the general theory to obtain the result.
We see this as a reflection of the homology divided difference operators of~\cite[Prop.~8.5]{LLS21} that act by adding boxes on diagonals (\textit{cf.}~\cite[Rem.~5.8]{Yel17}).
Additionally, the crystal structure is described using the usual description on semistandard tableau (see, \textit{e.g.},~\cite{BS18}) except for one ``correction'' when the result fails to be an edge labeled tableau.
Likewise, there exists an analog of the uncrowding algorithm~\cite{Buch02,HS20} that also uses the Robinson--Schensted--Knuth (RSK) algorithm.
This yields a crystal isomorphism with the recording tableau describing the element in the highest weight crystal $B(\lambda)$.

This paper is organised as follows:
In Section~\ref{sec:background}, we recall necessary tableaux and introduce the edge Schur functions.
In Section~\ref{sec:latticemodels}, interpret $\eschur^{\lambda}(\xx_n|\aaa)$ using an integrable lattice model and prove the skew Cauchy identity (Theorem~\ref{thm:cauchy}).
Finally, in Section~\ref{sec:crystal_structure} we study a crystal structure on $\elt$ to give a Schur expansion of $\eschur^{\lambda}$ and describe the uncrowding algorithm.

\subsection*{Acknowledgements}

The authors thank Jules Lamers and Paul Zinn-Justin for useful conversations.
The authors also thank Mark Shimozono, Hugh Thomas, Alexander Yong, and Paul Zinn-Justin for comments on earlier drafts of this manuscript.
AG thanks the University of Queensland for its hospitality during his visit in February, 2021.
The majority of these results were obtained while TS was affiliated with the School of Mathematics and Physics at the University of Queensland.
This work benefited from computations using {\sc SageMath}~\cite{sage,combinat}.
This work was partly supported by Osaka City University Advanced Mathematical Institute (MEXT Joint Usage/Research Center on Mathematics and Theoretical Physics JPMXP0619217849).

\section{Tableau generating functions}
\label{sec:background}

We give the necessary background on tableaux and generating functions.

Fix a positive integer $n$.
Denote $[n] := \{1, 2, \dotsc, n\}$.
A \defn{partition} $\lambda$ is a weakly decreasing finite sequence of positive integers, and we draw the Young diagram of $\lambda$ using English convention.
We will often extend $\lambda$ with a number (possibly infinite) of trailing zeros.
Let $\ell(\lambda)$ denote the length of $\lambda$, which is the index of the last positive entry of $\lambda$.
For partitions $\mu \subseteq \lambda$, a skew partition $\lambda / \mu$ is the set theoretical difference of the Young diagrams.
For a box $\bbb \in \lambda / \mu$ in row $i$ and column $j$, define the \defn{content} $c(\bbb) = j - i$.
Let $\xx = (x_1, x_2, \ldots)$ denote a countable sequence of indeterminates, and let $\xx_n = (x_1, x_2, \dotsc, x_n, 0, 0, \ldots)$ denote the specialization $x_i = 0$ for all $i > n$.
We define similar notation for other bold letters, for example $\yy = (y_1, y_2, \ldots)$ and $\yy_n = (y_1, y_2, \dotsc, y_n, 0, \ldots)$.
We make one exception for $\aaa = (\ldots, a_{-1}, a_0, a_1, \ldots)$.

\subsection{Factorial Schur functions}

For a skew shape $\lambda/\mu$, define a \defn{semi-standard tableau} as a filling of boxes of the Young diagram with positive integers such that entries weakly increase along rows (left to right) and strictly increase down columns (top to bottom).
Let $\ssyt_n(\lambda/\mu)$ denote the set of semi-standard tableaux of shape $\lambda / \mu$ with entries in $[n]$.
For every $T\in \ssyt_n(\lambda/\mu)$, define the \defn{weight} $\wt(T) = \prod_{i=1}^n x_i^{m_i}$, where $m_i$ is the number of times $i$ appears in $T$.
The \defn{reverse Far-Eastern reading word} is where we read the entries in each column from bottom-to-top and read the columns from left-to-right.
We say $T$ has the \defn{lattice property} if in the reverse Far-Eastern reading word $\sft = \sft_1 \sft_2 \dotsm \sft_{\abs{\lambda/\mu}}$ the number of $i$'s in any terminal word $\sft_k \dotsm \sft_{\abs{\lambda/\mu}}$ is at least as many as the number of $(i+1)$'s.
This is also known as the reading word being Yamanouchi.

For a skew tableau $T$, by abuse of notation, we let $\bbb \in T$ denote an entry in a box of $T$ and $c(\bbb)$ the content of the box (which does not depend on the entry).
A \defn{skew Schur polynomial} is defined as the generating series over such tableaux:
\[
     s_{\lambda/\mu}(\xx_n) = \sum_{T \in \ssyt_n(\lambda/\mu)} \prod_{\bbb \in T} x_{\bbb} = \sum_{T \in \ssyt_n(\lambda/\mu)} \wt(T),
\]
When $\mu$ is empty partition, we get an ordinary \defn{Schur polynomial}, which form a basis $\{s_{\lambda}(\xx_n)\}_{\lambda}$ (over all partitions $\lambda$) for the ring of symmetric polynomials $\ZZ[\xx_n]$, including in the $n \to \infty$ limit.
They are a orthonormal basis for the ring of symmetric polynomials under the \defn{Hall inner product}, which we take as our definition $\inner{s_{\lambda}(\xx_n)}{s_{\eta}(\xx_n)} = \delta_{\lambda\eta}$.
Hence, they satisfy the well known \defn{Cauchy identity} (see, \textit{e.g.},~\cite[Lemma~7.9.2]{ECII}):
\begin{equation}
\label{eq:cauchyforordinaryschur}
\sum_{\lambda} s_{\lambda}(\xx_n) s_{\lambda}(\yy_m)=\prod_{\substack{1\leq i\leq n,\\[0.2em]1\leq j\leq m}}\frac{1}{1-x_i y_j}.
\end{equation}
The self-duality implies a remarkable formula for the structure constants of $s_{\lambda}$.
We have
\[
    s_{\nu}(\xx_n) s_{\mu}(\xx_n) = \sum_{\lambda} c^{\lambda}_{\nu,\mu} s_{\lambda}(\xx_n),
\]
where $c^{\lambda}_{\nu,\mu}$ are the Littlewood--Richardson coefficients and are independent of $n$, and
\[
    s_{\lambda/\mu}(\xx_n) = \sum_{\nu} c^{\lambda}_{\nu,\mu} s_{\nu}(\xx_n),
\]
where the summation is over all partition $\nu \subseteq \lambda$ and $c^{\lambda}_{\nu,\mu}$ is equal to the number of $\ssyt$ of shape $\lambda/\nu$ with content $\mu$ whose reading word satisfies the lattice property.

The \defn{factorial Schur polynomial} is a generalization of $s_{\lambda/\mu}$ defined as
\[
s_{\lambda/\mu}(\xx | \aaa) = \sum_{T \in \ssyt_n(\lambda/\mu)} \prod_{\alpha \in T} (x_{\alpha} - a_{\alpha + c(\alpha)}).
\]
When $\aaa = 0$, we recover the skew Schur functions $s_{\lambda/\mu}(\xx_n) = s_{\lambda/\mu}(\xx_n | 0)$.

\begin{ex}
For partition $\lambda=(2,0)$ we have
\[
\begin{array}{rc@{\quad+\quad}c@{\quad+\quad}c}
s_{(2,0)}(\xx_2 | \aaa) = & \ytableaushort{11} & \ytableaushort{12} & \ytableaushort{22}\,,
\\[5pt]
 = & (x_1 - a_1) (x_1 - a_2) & (x_1 - a_1)(x_2 - a_3) & (x_2 - a_2)(x_2 - a_3).
 \end{array}
\]
\end{ex}

We need some definitions to state a formula for structure constants for $s_{\lambda}(\xx | \aaa)$.

\subsection{Edge Schur functions}

For a skew diagram $\nu/\lambda$, an \defn{edge labeled tableau} is a semistandard tableau where we allow each horizontal edge to also contain a finite set of positive integers, but we do not compare entries between two edges.
Therefore, the smallest of the set is strictly greater than the entry in the box above and the largest is strictly smaller than the entry in the box below.
We also allow edge labels to appear on the bottom row of $\mu$ which we consider to have a ``zeroth row'' of infinite length, so there are infinitely many edge labeled tableau of any skew shape.
We call the an entry in one of the sets on an edge an \defn{edge label}.
Let $\elt_n(\lambda/\mu)$ denote the set of edge labeled tableaux of shape $\lambda/\mu$ with entries and edge labels in $[n]$.

We now extend the definition of weight to an edge labeled tableau $T$.
Roughly speaking, the $\xx_n$ variables count the entries and the $\aaa$ variables count the diagonals the edge labels are located on.
More precisely, for $T\in \elt_n(\lambda/\mu)$, define
\[
\wt(T) = \prod_{\bbb \in T} x_{\bbb} \prod_{\ell \in A_{\bbb}} x_{\ell} a_{c(\bbb)},
\]
where $A_{\bbb}$ is the set labeling the \emph{upper} edge of the box $\bbb$.

\begin{ex}
An edge labeled tableau of shape $(3,2,2)$ and its weight is
\[
\ytableausetup{boxsize=1.5em}
T = \ytableaushort{1{\elb{1}{2}}{\elbb{124}{5}},{\elb{2}{3}}3,4{\elb{4}{5}}}\,,
\qquad\qquad
\begin{aligned}
\wt(T) & = x_1^2 x_2 x_3 x_4^2 x_5 (x_3 a_{-2}) (x_5 a_{-2}) (x_2 a_0) (x_1 a_2) (x_2 a_2) (x_4 a_2)
\\ & = a_{-2}^2 a_0 a_2^3 \cdot x_1^3 x_2^3 x_3^3 x_4^2 x_5^2.
\end{aligned}
\]
\end{ex}

\begin{ex}
\label{ex:skew_ELT}
An edge labeled tableau of shape $(5,3,2)/(4)$ is and weight is
\[
\ytableausetup{boxsize=1.5em}
T = \ytableaushort{{\cdot}{\elb{\cdot}{1}}{\cdot}{\elb{\cdot}{36}}{\elb{1}{25}},{\elb{1}{24}}{\elb{2}{3}}5,5{\elb{5}{6}}}\,,
\qquad\qquad
\wt(T) = a_{-2}^3 a_{-1} a_0 a_2^2 a_3^2 \cdot x_1^3 x_2^3 x_3^2 x_4^2 x_5^4 x_6^2.
\]
\end{ex}

\begin{dfn}[edge Schur functions]
We define the \defn{edge function} $\eschur^{\lambda/\mu}(\xx|\aaa)$ as a generating series of $\elt(\lambda/\mu)$:
\[
\eschur^{\lambda/\mu}(\xx_n | \aaa) = \sum_{T \in \elt_n(\lambda/\mu)} \wt(T).
\]
If $\mu \not\subseteq \lambda$, then define $\eschur^{\lambda/\mu}(\xx_n | \aaa) = 0$.
\end{dfn}

From the definition, it is clear that $\eschur^{\lambda/\mu}(\xx_n|\aaa)$ is an element of the formal power series ring $\ZZ[\xx_n | \aaa] := \ZZ[\xx_n][\![\aaa]\!]$ and are of the form $\eschur^{\lambda/\mu} = s_{\lambda/\mu} + HOT$, where $HOT$ denotes higher order terms with regards to the total degree (or just total degree in either $\xx_n$ or $\aaa$).
As a consequence, as we will show $\eschur^{\lambda}(\xx_n|\aaa)$ is a symmetric function, then $\{\eschur^{\lambda}\}_{\lambda}$ forms a basis since the change of basis is triangular with respect to the degree lexicographic order.
We will be implicitly working in a completion of the ring of symmetric functions similar to the case of (refined) Grothendieck functions, which does not alter the computations in any meaningful way.

%
%
%

\section{Lattice models}
\label{sec:latticemodels}

We shall use lattice model formulation to prove that $\eschur^{\lambda}(\xx | \aaa)$ are symmetric polynomials in the $\xx$ variables and our Cauchy identity.
To do so, we first describe the standard diagrammatic formalism and the interpretation as linear operators.
Next, we reinterpret $\eschur^{\lambda}(\xx|\aaa)$ as the partition function of an integrable lattice model (in the sense that there is a solution to the Yang--Baxter equation) through a direct translation of the combinatorial information.
Using a known integrable five-vertex model for the factorial Schur polynomials, we show that pairing this with a dual version of our lattice model for $\eschur^{\lambda}(\xx|\aaa)$ also has a solution to the Yang--Baxter equation.
Our proof concludes with the train argument.

\subsection{Conventions}

The data for a \defn{lattice model} $\fM$ (also known as a \defn{vertex model}) consists of
\begin{itemize}
\item a subset $G$ of the square grid on $\ZZ_{>0} \times \ZZ$ called vertices,
\item parameters $\xx$ and $\aaa$ that are assigned to each horizontal and vertical edge, respectively,
\item a set of edge labels $X$, and
\item a \defn{Boltzmann weight} function $\Bwt \colon \mcC \to \mathcal{R}$ that maps possible configurations around each vertex $\mcC$ to a fixed ring $\mathcal{R}$ (which can depend on the edge parameters).
\end{itemize}
We consider the horizontal (resp.\ vertical) lines as being oriented left to right (resp.\ bottom to top).
The \defn{boundary} will be the edges that only have one endpoint in $G$, and we will often associate with a model a fixed set of boundary conditions.
A \defn{state} $\mcS$ of $\fM$ will be an assignment of labels to all edges with at least one endpoint in $G$, and the Boltzmann weight of $\mcS$ is $\Bwt(\mcS) = \prod_{v \in G} \Bwt(v)$ the product of Boltzmann weights of each local configuration around each vertex.
The \defn{partition function} of $\fM$ is $Z(\fM) = \sum_{\mcS \in \fM} \Bwt(\mcS)$ the sum of the Boltzmann weights of the states of $\fM$.

In this paper, will only consider $G = [n] \times [m, M]$ with possibly $n,m,M$ to be infinite.
Additionally, we use $X = \{0,1\}$, which means there are 16 possible inputs for the Boltzmann weight function.
We will impose the square-ice condition, which means that anytime the sum of the incoming edges does not equal the sum of the outgoing edges, the Boltzmann weight of this configuration must be $0$.
This means we have at most six possible configurations with a potentially nonzero Boltzmann weight.
We will think of a $0$ as being a hole and a $1$ as being a particle, and the square-ice condition is a conservation of particles.

Furthermore, given our choice of $X$, we can attach a two dimensional vector spaces
\[
H_i \iso \CC^{2} = \CC\{\epsilon_0, \epsilon_1\},
\qquad\qquad V_j \iso \CC^2 = \CC\{\ket{0}, \ket{1}\},
\]
to $i$-th horizontal line (counted from the bottom) called the \defn{quantum space} and the $j$-th vertical line called the \defn{physical space}, respectively.
We reformulate the Boltzmann weights as a linear operator called \defn{$L$-matrix} $L \colon H_i\otimes V_j \to H_i\otimes V_j$.
We note an important convention choice in the way we write our matrices.
The ordered basis we will use for $H_i \otimes V_j$ is
\[
(\epsilon_0 \otimes \ket{0}, \epsilon_0 \otimes \ket{1}, \epsilon_1 \otimes \ket{0}, \epsilon_1 \otimes \ket{1}).
\]
Hence, we draw the $L$-matrix graphically as
\[
\begin{pmatrix}
 \label{random_lmat}
\begin{tikzpicture}[scale=0.6,baseline=4pt,rotate=45]
\draw (0,0)--(1,1);
\draw (0,1)--(1,0);
\ebull{0}{1}{0.1};
\ebull{1}{0}{0.1};
\ebull{1}{1}{0.1};
\ebull{0}{0}{0.1};
\end{tikzpicture}& 0 & 0 & 0 \\[1ex]
0 & \begin{tikzpicture}[scale=0.6,baseline=4pt,rotate=45]
\draw (0,0)--(1,1);
\draw (0,1)--(1,0);
\bbull{0}{0}{0.1};
\ebull{0}{1}{0.1};
\ebull{1}{0}{0.1};
\bbull{1}{1}{0.1};
\end{tikzpicture} & \begin{tikzpicture}[scale=0.6,baseline=4pt,rotate=45]
\draw (0,0)--(1,1);
\draw (0,1)--(1,0);
\ebull{0}{1}{0.1};
\ebull{1}{1}{0.1};
\bbull{1}{0}{0.1};
\bbull{0}{0}{0.1};
\end{tikzpicture} & 0 \\[1ex]
0 &\begin{tikzpicture}[scale=0.6,baseline=4pt,rotate=45]
\draw (0,0)--(1,1);
\draw (0,1)--(1,0);
\bbull{0}{1}{0.1};
\ebull{0}{0}{0.1};
\ebull{1}{0}{0.1};
\bbull{1}{1}{0.1};
\end{tikzpicture} & \begin{tikzpicture}[scale=0.6,baseline=4pt,rotate=45]
\draw (0,0)--(1,1);
\draw (0,1)--(1,0);
\bbull{1}{0}{0.1};
\bbull{0}{1}{0.1};
\ebull{1}{1}{0.1};
\ebull{0}{0}{0.1};
\end{tikzpicture} & 0 \\[1ex]
0 & 0 & 0 & \begin{tikzpicture}[scale=0.6,baseline=4pt,rotate=45]
\draw (0,0)--(1,1);
\draw (0,1)--(1,0);
\bbull{1}{0}{0.1};
\bbull{0}{1}{0.1};
\bbull{1}{1}{0.1};
\bbull{0}{0}{0.1};
\end{tikzpicture} \\
\end{pmatrix}_{ij}
\in \End(H_i \otimes V_j).
\]

Let us look at the subset of all physical space $\mcV \subseteq \bigotimes_{j \in \ZZ} V_j$ such that for any $\cdots \ket{v_{-\frac{1}{2}}}_{-1} \otimes \ket{v_{\frac{1}{2}}}_0 \otimes \ket{v_{\frac{3}{2}}}_1 \otimes \cdots \in \mcV$, we have $v_{-j+\frac{1}{2}} = 1$ and $v_{j+\frac{1}{2}} = 0$ for all $j \gg 1$ and $\sum_{j \in [-m,m]} v_j = m$ for all $m \gg 1$.
We can identify elements in $\mcV$ with partitions by using the \defn{Maya diagram} (or \defn{$01$-sequence}) of a partition $\lambda$ via the well-known correspondence of vertical (resp.\ horizontal) steps of the boundary of the Young diagram being $1$ (resp.~$0$).
We denote the corresponding vector by $\ket{\lambda} \in \mcV$.
We have shifted the indices for $v_i$ to be half integers to match the midpoints of the edges from the Maya diagram of $\lambda$, and so the empty partition, also called the \defn{vacuum element}, corresponds to when $v_j = 1$ for $j < 0$ and $v_j = 0$ for $j \geq 0$
or $\ket{\emptyset} = \cdots \otimes  \ket{1}_{-2} \otimes \ket{1}_{-1} \otimes \ket{0}_0 \otimes \ket{0}_1 \otimes \cdots$.

\begin{ex}
The partition $\lambda = (3,3,1)$ is
\[
\begin{tikzpicture}[scale=0.8,rotate=-45,>=latex]
\draw (0,0) grid (3,-2);
\draw (0,-2) grid (1,-3);
\draw[dashed,->] (5,0) -- (7,0);
\draw(3,0)--(5,0);
\draw(4,0)--(4,-0.2);
\draw(5,0)--(5,-0.2);
\draw(0,-3)--(0,-4);
\draw(0,-4)--(0.2,-4);
\draw[dashed,->] (0,-4) -- (0,-6);
\foreach \x/\h in {0/2.5,1/2.5,3/2.5} {
  \foreach\y in {0,0.25,...,\h} {
    \bbull{\x+\y}{-3.5+\x-\y}{0.07};
    \ebull{\x+.5+\y}{-3+\x-\y}{0.07};
  }
}
\foreach\y in {0.5,0.75,...,2.5}
  \bbull{2.5+\y}{-1-\y}{0.07};
\foreach\y in {0.5,0.75,...,2.5}
  \ebull{2+\y}{-1.5-\y}{0.07};
\foreach\y in {0.5,0.75,...,2.5}
  \ebull{4+\y}{0.5-\y}{0.07};
\draw[thick,<->](1.5,-8)--(8.5,-1);
\ebull{7}{-2.5}{0.1};
\ebull{6.5}{-3}{0.1};
\bbull{6}{-3.5}{0.1};
\bbull{5.5}{-4}{0.1};
\ebull{5}{-4.5}{0.1};
\ebull{4.5}{-5}{0.1};
\bbull{4}{-5.5}{0.1};
\ebull{3.5}{-6}{0.1};
\bbull{3}{-6.5}{0.1};
\foreach\x in{0,0.5,1,1.5,2,2.5,3,3.5,4,4.5}
{\draw(2.75+\x,-6.75+\x)--(2.6+\x,-6.6+\x);}
\end{tikzpicture}
\]
and corresponds to the vector in $\mcV$ is
\[
\ket{\lambda} = \cdots \otimes \ket{1}_{-4} \otimes \ket{0} _{-3} \otimes\ket{1}_{-2} \otimes \ket{0}_{-1} \otimes \ket{0}_0 \otimes \ket{1}_1 \otimes \ket{1}_2 \otimes \ket{0}_3 \otimes \ket{0}_4 \otimes \cdots.
\]
\end{ex}

If we consider $\lambda$ to be contained in a $k \times m$ rectangle, we can restrict to a finite binary string of length $(k+m)$ with exactly $k$ $1$'s and $m$ $0$'s.
If we restrict to $\ell(\lambda) \leq k$, then we can represent a partition as an (semi) infinite binary string with exactly $k$ $1$'s.

\begin{ex}
The binary string $1010100100$ is associated to the partition $\lambda = (4,2,1,0)$ contained in $4\times 6$ rectangle:
\[
\begin{tikzpicture}[scale=0.6]
\draw[fill=gray] (0,3) rectangle (4,2);
\draw[fill=gray] (0,2) rectangle (2,1);
\draw[fill=gray] (0,0) rectangle (1,1);
\draw (0,-1) grid (6,3);
\node at (5.5,2.7) {$\ss 0$};
\node at (4.5,2.7) {$\ss 0$};
\node at (4.2,2.3) {$\ss 1$};
\node at (3.5,1.7) {$\ss 0$};
\node at (2.5,1.7) {$\ss 0$};
\node at (2.2,1.3) {$\ss 1$};
\node at (1.5,0.7) {$\ss 0$};
\node at (1.2,0.3) {$\ss 1$};
\node at (0.5,-0.3) {$\ss 0$};
\node at (0.2,-0.7) {$\ss 1$};
\end{tikzpicture}
\]
\end{ex}

Throughout this paper, we shall move freely back and forth between partitions, binary strings when the rectangle is clear, and the vectors $\ket{\lambda}$.

\subsection{Lattice model for edge Schur functions}

In what follows, we give a lattice model whose partition function is $\eschur^{\lambda}(\xx_n|\aaa)$.
Our convention for drawing is a product of matrices read right-to-left is drawn from southwest to northeast.

Define the $L$-matrix with the nonzero Boltzmann weights:
\begin{equation}
\label{eq:dual_schur_vertices}
\begin{array}{c*{4}{@{\hspace{40pt}}c}}
\textsf{a}_1 & \textsf{b}_1 & \textsf{b}_2 & \textsf{c}_1 & \textsf{c}_2
\\
\begin{tikzpicture}[scale=0.4,baseline=-2pt]
\draw[arrow=0.25,dscol] (-1,0) node[left,black] {$\ss {0}$} -- (1,0) node[right,black] {$\ss {0}$};
\draw[arrow=0.25] (0,-1) node[below] {$\ss {0}$} -- (0,1) node[above] {$\ss{0}$};
\end{tikzpicture}
&
\begin{tikzpicture}[scale=0.4,baseline=-2pt]
\draw[arrow=0.25,dscol] (-1,0) node[left,black] {$\ss {0}$} -- (1,0) node[right,black] {$\ss {0}$};
\draw[arrow=0.25] (0,-1) node[below] {$\ss {1}$} -- (0,1) node[above] {$\ss {1}$};
\end{tikzpicture}
&
\begin{tikzpicture}[scale=0.4,baseline=-2pt]
\draw[arrow=0.25,dscol] (-1,0) node[left,black] {$\ss {1}$} -- (1,0) node[right,black] {$\ss {1}$};
\draw[arrow=0.25] (0,-1) node[below] {$\ss {0}$} -- (0,1) node[above] {$\ss {0}$};
\end{tikzpicture}
&
\begin{tikzpicture}[scale=0.4,baseline=-2pt]
\draw[arrow=0.25,dscol] (-1,0) node[left,black] {$\ss {0}$} -- (1,0) node[right,black] {$\ss {1}$};
\draw[arrow=0.25] (0,-1) node[below] {$\ss {1}$} -- (0,1) node[above] {$\ss {0}$};
\end{tikzpicture}
&
\begin{tikzpicture}[scale=0.4,baseline=-2pt]
\draw[arrow=0.25,dscol] (-1,0) node[left,black] {$\ss {1}$} -- (1,0) node[right,black] {$\ss {0}$};
\draw[arrow=0.25] (0,-1) node[below] {$\ss {0}$} -- (0,1) node[above] {$\ss {1}$};
\end{tikzpicture}
\\
1 + a_j x_i & 1 & x_i & 1 & x_i
\end{array}
\end{equation}
which is given as the linear operator $L_{ij} \colon H_i \otimes V_j \to H_i \otimes V_j$ by
\[
 L_{ij}(x_i, a_j) 
=
\begin{pmatrix}
1+a_j x_i & 0 & 0 & 0 \\[1ex]
0 & 1 & 1 & 0 \\[1ex]
0 & x_i & x_i & 0 \\[1ex]
0 & 0 & 0 & 0 \\
\end{pmatrix}_{ij}
=
\begin{pmatrix}
\begin{tikzpicture}[scale=0.6,baseline=4pt,rotate=45]
\draw(0,0)--(1,1);
\draw[dscol] (0,1)--(1,0);
\ebull{0}{1}{0.1};
\ebull{1}{0}{0.1};
\ebull{1}{1}{0.1};
\ebull{0}{0}{0.1};
\end{tikzpicture}& 0 & 0 & 0 \\[1ex]
0 & \begin{tikzpicture}[scale=0.6,baseline=4pt,rotate=45]
\draw(0,0)--(1,1);
\draw[dscol] (0,1)--(1,0);
\bbull{0}{0}{0.1};
\ebull{0}{1}{0.1};
\ebull{1}{0}{0.1};
\bbull{1}{1}{0.1};
\end{tikzpicture} & \begin{tikzpicture}[scale=0.6,baseline=4pt,rotate=45]
\draw (0,0)--(1,1);
\draw[dscol] (0,1)--(1,0);
\ebull{0}{1}{0.1};
\ebull{1}{1}{0.1};
\bbull{1}{0}{0.1};
\bbull{0}{0}{0.1};
\end{tikzpicture} & 0 \\[1ex]
0 &\begin{tikzpicture}[scale=0.6,baseline=4pt,rotate=45]
\draw (0,0)--(1,1);
\draw[dscol] (0,1)--(1,0);
\bbull{0}{1}{0.1};
\ebull{0}{0}{0.1};
\ebull{1}{0}{0.1};
\bbull{1}{1}{0.1};
\end{tikzpicture} & \begin{tikzpicture}[scale=0.6,baseline=4pt,rotate=45]
\draw (0,0)--(1,1);
\draw[dscol] (0,1)--(1,0);
\bbull{1}{0}{0.1};
\bbull{0}{1}{0.1};
\ebull{1}{1}{0.1};
\ebull{0}{0}{0.1};
\end{tikzpicture} & 0 \\[1ex]
0 & 0 & 0 & \begin{tikzpicture}[scale=0.6,baseline=4pt,rotate=45]
\draw (0,0)--(1,1);
\draw[dscol] (0,1)--(1,0);
\bbull{1}{0}{0.1};
\bbull{0}{1}{0.1};
\bbull{1}{1}{0.1};
\bbull{0}{0}{0.1};
\end{tikzpicture} \\
\end{pmatrix}_{ij}
\in 
\End(H_i \otimes H_j).
\]
It is convenient to define a dual model where the vertices that are obtained from~\eqref{eq:dual_schur_vertices} by flipping the vertices upside down and interchanging $0 \leftrightarrow 1$ on the horizontal edges.
Thus, we define the $L^*$-matrix as
\begin{gather*}
\begin{array}{c*{4}{@{\hspace{40pt}}c}}
\begin{tikzpicture}[scale=0.4,baseline=-2pt]
\draw[arrow=0.25,ddscol] (-1,0) node[left,black] {$\ss {1}$} -- (1,0) node[right,black] {$\ss {1}$};
\draw[arrow=0.25] (0,-1) node[below] {$\ss {0}$} -- (0,1) node[above] {$\ss{0}$};
\end{tikzpicture}
&
\begin{tikzpicture}[scale=0.4,baseline=-2pt]
\draw[arrow=0.25,ddscol] (-1,0) node[left,black] {$\ss {1}$} -- (1,0) node[right,black] {$\ss {1}$};
\draw[arrow=0.25] (0,-1) node[below] {$\ss {1}$} -- (0,1) node[above] {$\ss {1}$};
\end{tikzpicture}
&
\begin{tikzpicture}[scale=0.4,baseline=-2pt]
\draw[arrow=0.25,ddscol] (-1,0) node[left,black] {$\ss {0}$} -- (1,0) node[right,black] {$\ss {0}$};
\draw[arrow=0.25] (0,-1) node[below] {$\ss {0}$} -- (0,1) node[above] {$\ss {0}$};
\end{tikzpicture}
&
\begin{tikzpicture}[scale=0.4,baseline=-2pt]
\draw[arrow=0.25,ddscol] (-1,0) node[left,black] {$\ss {1}$} -- (1,0) node[right,black] {$\ss {0}$};
\draw[arrow=0.25] (0,-1) node[below] {$\ss {0}$} -- (0,1) node[above] {$\ss {1}$};
\end{tikzpicture}
&
\begin{tikzpicture}[scale=0.4,baseline=-2pt]
\draw[arrow=0.25,ddscol] (-1,0) node[left,black] {$\ss {0}$} -- (1,0) node[right,black] {$\ss {1}$};
\draw[arrow=0.25] (0,-1) node[below] {$\ss {1}$} -- (0,1) node[above] {$\ss {0}$};
\end{tikzpicture}
\\
1 + a_j x_i & 1 & x_i & 1 & x_i
\end{array}
\\
L^{*}_{ij}(x_i, a_j) =
\begin{pmatrix}
x_i & 0 & 0 & 0 \\[1ex]
0 & 0 & x_i & 0 \\[1ex]
0 & 1 & 1+a_j x_i & 0 \\[1ex]
0 & 0 & 0 & 1 \\
\end{pmatrix}_{ij}
=
 \begin{pmatrix}
\begin{tikzpicture}[scale=0.6,baseline=4pt,rotate=45]
\draw(0,0)--(1,1);
\draw[ddscol] (0,1)--(1,0);
\ebull{0}{1}{0.1};
\ebull{1}{0}{0.1};
\ebull{1}{1}{0.1};
\ebull{0}{0}{0.1};
\end{tikzpicture}& 0 & 0 & 0 \\[1ex]
0 & \begin{tikzpicture}[scale=0.6,baseline=4pt,rotate=45]
\draw(0,0)--(1,1);
\draw[ddscol] (0,1)--(1,0);
\bbull{0}{0}{0.1};
\ebull{0}{1}{0.1};
\ebull{1}{0}{0.1};
\bbull{1}{1}{0.1};
\end{tikzpicture} & \begin{tikzpicture}[scale=0.6,baseline=4pt,rotate=45]
\draw (0,0)--(1,1);
\draw[ddscol] (0,1)--(1,0);
\ebull{0}{1}{0.1};
\ebull{1}{1}{0.1};
\bbull{1}{0}{0.1};
\bbull{0}{0}{0.1};
\end{tikzpicture} & 0 \\[1ex]
0 &\begin{tikzpicture}[scale=0.6,baseline=4pt,rotate=45]
\draw (0,0)--(1,1);
\draw[ddscol] (0,1)--(1,0);
\bbull{0}{1}{0.1};
\ebull{0}{0}{0.1};
\ebull{1}{0}{0.1};
\bbull{1}{1}{0.1};
\end{tikzpicture} & \begin{tikzpicture}[scale=0.6,baseline=4pt,rotate=45]
\draw (0,0)--(1,1);
\draw[ddscol] (0,1)--(1,0);
\bbull{1}{0}{0.1};
\bbull{0}{1}{0.1};
\ebull{1}{1}{0.1};
\ebull{0}{0}{0.1};
\end{tikzpicture} & 0 \\[1ex]
0 & 0 & 0 & \begin{tikzpicture}[scale=0.6,baseline=4pt,rotate=45]
\draw (0,0)--(1,1);
\draw[ddscol] (0,1)--(1,0);
\bbull{1}{0}{0.1};
\bbull{0}{1}{0.1};
\bbull{1}{1}{0.1};
\bbull{0}{0}{0.1};
\end{tikzpicture}
\allowdisplaybreaks \\
\end{pmatrix}_{ij}
\in 
\End(H_i \otimes V_j).
\end{gather*}
It will sometimes be convenient to consider $L^*_{ij} \in \End(H_i^* \otimes V_j)$ in the sequel.
We refer the vertices
$\begin{tikzpicture}[scale=0.6,baseline=4pt,rotate=45]
\draw (0,0)--(1,1);
\draw[ddscol] (0,1)--(1,0);
\bbull{1}{0}{0.1};
\bbull{0}{1}{0.1};
\ebull{1}{1}{0.1};
\ebull{0}{0}{0.1};
\end{tikzpicture}$
and
$\begin{tikzpicture}[scale=0.6,baseline=4pt,rotate=45]
\draw (0,0)--(1,1);
\draw[dscol] (0,1)--(1,0);
\ebull{1}{0}{0.1};
\ebull{0}{1}{0.1};
\ebull{1}{1}{0.1};
\ebull{0}{0}{0.1};
\end{tikzpicture}$
as the \defn{deformed vertices}.

\begin{prop}
This vertex model is \defn{integrable}, which means there exists an \defn{$R$-matrix}
\[
 R_{ij}\left( \frac{x_j}{x_i} \right)=
 \begin{pmatrix}
\begin{tikzpicture}[scale=0.6,baseline=4pt]
\draw[dscol] (0,0)--(1,1);
\draw[dscol] (0,1)--(1,0);
\ebull{0}{1}{0.1};
\ebull{1}{0}{0.1};
\ebull{1}{1}{0.1};
\ebull{0}{0}{0.1};
\end{tikzpicture}& 0 & 0 & 0 \\[1ex]
0 & \begin{tikzpicture}[scale=0.6,baseline=4pt]
\draw[dscol] (0,0)--(1,1);
\draw[dscol] (0,1)--(1,0);
\bbull{0}{0}{0.1};
\ebull{0}{1}{0.1};
\ebull{1}{0}{0.1};
\bbull{1}{1}{0.1};
\end{tikzpicture} & \begin{tikzpicture}[scale=0.6,baseline=4pt]
\draw[dscol] (0,0)--(1,1);
\draw[dscol] (0,1)--(1,0);
\ebull{0}{1}{0.1};
\ebull{1}{1}{0.1};
\bbull{1}{0}{0.1};
\bbull{0}{0}{0.1};
\end{tikzpicture} & 0 \\[1ex]
0 &\begin{tikzpicture}[scale=0.6,baseline=4pt]
\draw[dscol] (0,0)--(1,1);
\draw[dscol] (0,1)--(1,0);
\bbull{0}{1}{0.1};
\ebull{0}{0}{0.1};
\ebull{1}{0}{0.1};
\bbull{1}{1}{0.1};
\end{tikzpicture} & \begin{tikzpicture}[scale=0.6,baseline=4pt]
\draw[dscol] (0,0)--(1,1);
\draw[dscol] (0,1)--(1,0);
\bbull{1}{0}{0.1};
\bbull{0}{1}{0.1};
\ebull{1}{1}{0.1};
\ebull{0}{0}{0.1};
\end{tikzpicture} & 0 \\[1ex]
0 & 0 & 0 & \begin{tikzpicture}[scale=0.6,baseline=4pt]
\draw[dscol] (0,0)--(1,1);
\draw[dscol] (0,1)--(1,0);
\bbull{1}{0}{0.1};
\bbull{0}{1}{0.1};
\bbull{1}{1}{0.1};
\bbull{0}{0}{0.1};
\end{tikzpicture} \\
\end{pmatrix}_{ij}
=
\begin{pmatrix}
1 & 0 & 0 & 0 \\[1ex]
0 & 0 & \frac{x_j}{x_i} & 0 \\[1ex]
0 & 1 & 1-\frac{x_j}{x_i} & 0 \\[1ex]
0 & 0 & 0 & \frac{x_j}{x_i} \\
\end{pmatrix}_{ij}
\in 
\End(H_i \otimes H_j)
\]
that satisfies the $RLL$ form of the \defn{Yang--Baxter equation} in $\End(H_i \otimes H_j \otimes V_k)$:
\begin{align*}
L_{jk}(x_j,a_k) L_{ik}(x_i,a_k) R_{ij}(x_j/x_i) & = R_{ij}(x_j/x_i) L_{ik}(x_i,a_k) L_{jk}(x_j,a_k),
\\
L^*_{jk}(x_j,a_k) L^*_{ik}(x_i,a_k) R_{ij}(x_i/x_j) & = R_{ij}(x_i/x_j) L^*_{ik}(x_i,a_k) L^*_{jk}(x_j,a_k),
\end{align*}
which are represented pictorially, respectively, as
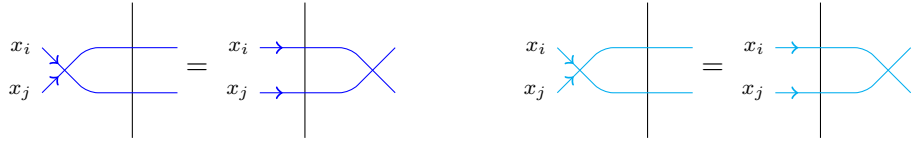
\[
\begin{tikzpicture}[scale=0.6,baseline=7pt]
\draw(1,-1)--(1,2);
\draw[rounded corners,arrow=0.15,dscol] (-1,1) node[left,black]{$\ss x_i$}--(0,0)--(2,0);
\draw[rounded corners,arrow=0.15,dscol] (-1,0)  node[left,black]{$\ss x_j$}--(0,1)--(2,1);
\end{tikzpicture}
=
\begin{tikzpicture}[scale=0.6,baseline=7pt]
\draw (1,2)--(1,-1);
\draw[arrow=0.15,rounded corners,dscol] (0,0) node[left,black]{$\ss x_j$}--(2,0)--(3,1);
\draw[arrow=0.15,rounded corners,dscol] (0,1)node[left,black]{$\ss x_i$}--(2,1)--(3,0);
\end{tikzpicture}
\qquad\qquad
\begin{tikzpicture}[scale=0.6,baseline=7pt]
\draw(1,-1)--(1,2);
\draw[rounded corners,arrow=0.15,ddscol] (-1,1) node[left,black]{$\ss x_i$}--(0,0)--(2,0);
\draw[rounded corners,arrow=0.15,ddscol] (-1,0)  node[left,black]{$\ss x_j$}--(0,1)--(2,1);
\end{tikzpicture}
=
\begin{tikzpicture}[scale=0.6,baseline=7pt]
\draw (1,2)--(1,-1);
\draw[arrow=0.15,rounded corners,ddscol] (0,0) node[left,black]{$\ss x_j$}--(2,0)--(3,1);
\draw[arrow=0.15,rounded corners,ddscol] (0,1)node[left,black]{$\ss x_i$}--(2,1)--(3,0);
\end{tikzpicture}
\]
\end{prop}

\begin{proof}
This is a finite computation.
Pictorially, this can be checked by looking at any fixed boundary conditions and seeing the partition functions are equal on both sides.
\end{proof}

We define a one row \defn{transfer matrix} $\TMat(x_i) \colon \mcV \to \mcV$ as a linear operator acting on $H_i \otimes \mcV$ that is expressed as a product of $L$-matrices:
\[
\TMat(x_i) = \prod_{j \in \ZZ}L_{ij}(x_i,a_j)
\qquad
\left(
\begin{tikzpicture}[scale=0.6,baseline=-1pt]
\draw[arrow=0.05,dscol] (-1,0)node[left,black]{} -- (7,0) node[right,black] {};
\foreach\x in {0,1,...,6}{
\draw[arrow=0.25] (\x,-1) -- (\x,1);
}
\node[below] at (0,-0.8) {$\cdots$};\node[above] at (0,0.8) {$\cdots$};
\node[below] at (1,-0.8) {\tiny $k_{-2}$};\node[above] at (1,0.8) {\tiny $k'_{-2}$};
\node[below] at (2,-0.8) {\tiny $k_{-1}$};\node[above] at (2,0.8) {\tiny $k'_{-1}$};
\node[below] at (3,-0.8) {\tiny $k_0$};\node[above] at (3,0.8) {\tiny $k'_0$};
\node[below] at (4,-0.8) {\tiny $k_1$};\node[above] at (4,0.8) {\tiny $k'_1$};
\node[below] at (5,-0.8) {\tiny $k_2$};\node[above] at (5,0.8) {\tiny $k'_2$};
\node[below] at (6,-0.8) {$\cdots$}; \node[above] at (6,0.8) {$\cdots$};
\end{tikzpicture}
\right).
\]
Using $L^*$-matrices, we define the dual transfer matrix $\TMat^{*}(x_i) \colon \mcV \to \mcV$ as
\[
\TMat^{*}(x_i) = \prod_{j \in \ZZ} L^*_{ij}(x_i, a_j)
\qquad
\left(
\begin{tikzpicture}[scale=0.6,baseline=-1pt]
\draw[arrow=0.05,ddscol] (-1,0)node[left,black]{} -- (7,0) node[right,black] {};
\foreach\x in {0,1,...,6}{
\draw[arrow=0.25] (\x,-1) -- (\x,1);
}
\node[below] at (0,-0.8) {$\cdots$};\node[above] at (0,0.8) {$\cdots$};
\node[below] at (1,-0.8) {\tiny $k_{-2}$};\node[above] at (1,0.8) {\tiny $k'_{-2}$};
\node[below] at (2,-0.8) {\tiny $k_{-1}$};\node[above] at (2,0.8) {\tiny $k'_{-1}$};
\node[below] at (3,-0.8) {\tiny $k_0$};\node[above] at (3,0.8) {\tiny $k'_0$};
\node[below] at (4,-0.8) {\tiny $k_1$};\node[above] at (4,0.8) {\tiny $k'_1$};
\node[below] at (5,-0.8) {\tiny $k_2$};\node[above] at (5,0.8) {\tiny $k'_2$};
\node[below] at (6,-0.8) {$\cdots$}; \node[above] at (6,0.8) {$\cdots$};
\end{tikzpicture}
\right).
\]

For any fixed partitions $\lambda$ and $\mu$, the transfer matrices have precisely one state with the top and bottom boundaries being $\lambda$ and $\mu$ respectively.
For the transfer matrix $\TMat$, if we write $\TMat(y_i) \ket{\mu} = c_{\mu}^{\lambda} \ket{\lambda}$ for $\ket{\mu},\ket{\lambda} \in \mcV$, then $c_{\mu}^{\lambda} \in \ZZ[y_i|\aaa]$ is a finite sum of states (although the weight of such a state will be an infinite product).
Indeed, note that since all vertical labels are $1$ for any index at most $-\ell(\lambda)$, there is only one possible configuration that forces the corresponding horizontal labels to be $0$.
Hence, for any index strictly less than $-\ell(\lambda)$, the Boltzmann weight is $1$, and so we can freely disregard these states and consider only a semi-infinite tensor product.
Since the number of particles must be conserved and that the vertical labels are $0$ for any index at least $\lambda_1$, there is also only one possible configuration for vertices sufficiently far to the right.
In the sequel, we will also set $a_i = 0$ for $i \gg 1$, and in this case we have a finite product and can restrict to a finite grid.
Furthermore, we can realize the transfer matrix as the $k$-th particle moving from the position $\mu_k - k$ to $\lambda_k - k$.
The case is similar for $\TMat^*$.

We introduce some other notation for convenience.
Let $\bra{\lambda} \in \mcV^*$ denote the dual basis vector to $\ket{\lambda}$ such that $\braket{\lambda}{\mu} = \bra{\lambda}(\ket{\mu}) = \delta_{\lambda\mu}$.
Thus, we have
\begin{align*}
Z(\fE_{\lambda/\mu}) & = \bra{\lambda} \TMat(x_n) \dotsm \TMat(x_2) \TMat(x_1) \ket{\mu},
\\
Z(\fE^*_{\lambda/\mu}) & = \bra{\mu} \TMat^*(x_1) \TMat^*(x_2) \dotsm \TMat^*(x_n) \ket{\lambda}.
\end{align*}
We note that $Z(\fE_{\lambda/\mu}) = Z(\fE^*_{\lambda/\mu})$ by the construction of the dual model.

\begin{ex}
\label{ex:transfer_state}
Let $\lambda = (4,4,1)$ and $\mu = (4,2)$.
We construct the unique state in $\bra{\lambda} \TMat(x) \ket{\mu}$:
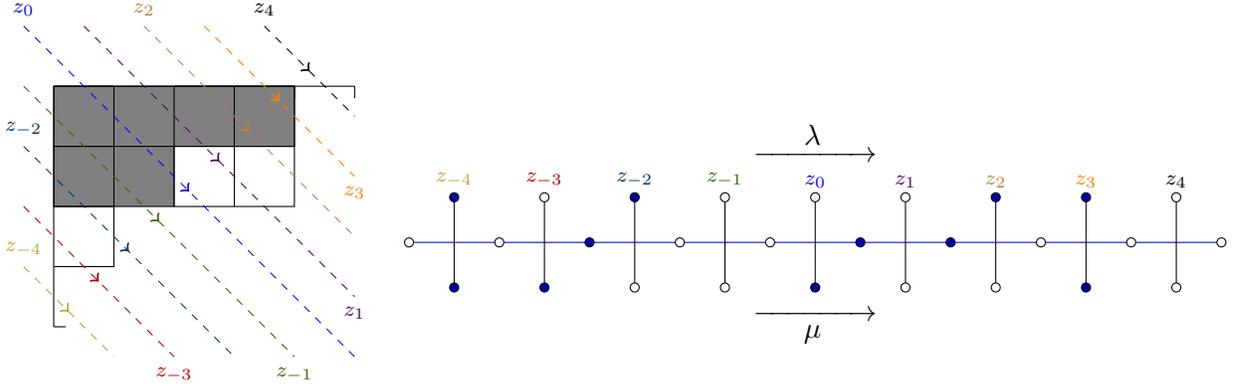
\[
\begin{tikzpicture}[scale=0.8]
\draw[fill=gray] (0,0) rectangle (4,-1);
\draw[fill=gray] (0,-1) rectangle (2,-2);
\draw (0,0) grid (1,-3);
\draw (2,0) grid (4,-2);
\draw (0,0) -- (5,0);
\draw(4,0)--(4,-0.2);
\draw(5,0)--(5,-0.2);
\draw(0,-3)--(0,-4);
\draw(0,-4)--(0.2,-4);
\draw[dashed,arrow,UQgold] (-0.5,-3)node[above]{$\ss z_{-4}$}--(1,-4.5) ;
\draw[dashed,arrow,darkred] (-0.5,-2)--(2,-4.5) node[below]{$\ss z_{-3}$} ;
\draw[dashed,arrow,OCUsapphire] (-0.5,-1)node[above]{$\ss z_{-2}$}--(3,-4.5);
\draw[dashed,arrow,dgreencolor] (-0.5,0)--(4,-4.5)node[below]{$\ss z_{-1}$};
\draw[dashed,arrow,blue] (-0.5,1)node[above]{$\ss z_{0}$}--(5,-4.5);
\draw[dashed,arrow,UQpurple] (0.5,1)--(5,-3.5)node[below]{$\ss z_{1}$};
\draw[dashed,arrow,brown] (1.5,1)node[above]{$\ss z_{2}$}--(5,-2.5);
\draw[dashed,arrow,orange] (2.5,1)--(5,-1.5)node[below]{$\ss z_{3}$};
\draw[dashed,arrow,black] (3.5,1)node[above]{$\ss z_{4}$}--(5,-0.5);
\end{tikzpicture}
\quad
\begin{tikzpicture}[scale=0.6,baseline=-5pt]
\draw[dscol] (1,3) -- (19,3);
\ebull{1}{3}{0.1};
\ebull{19}{3}{0.1};
\draw (2,2) -- (2,4)node[above,UQgold]{$\ss z_{-4}$};
\draw (4,2) -- (4,4)node[above,darkred]{$\ss z_{-3}$};
\draw (6,2) -- (6,4)node[above,OCUsapphire]{$\ss z_{-2}$};
\draw (8,2) -- (8,4)node[above,dgreencolor]{$\ss z_{-1}$};
\draw (10,2) -- (10,4)node[above,blue]{$\ss z_{0}$};
\draw (12,2) -- (12,4)node[above,UQpurple]{$\ss z_{1}$};
\draw (14,2) -- (14,4)node[above,brown]{$\ss z_{2}$};
\draw (16,2) -- (16,4)node[above,orange]{$\ss z_{3}$};
\draw (18,2) -- (18,4)node[above,black]{$\ss z_{4}$};
\bbull{2}{2}{0.1};
\bbull{4}{2}{0.1};
\ebull{6}{2}{0.1};
\ebull{8}{2}{0.1};
\bbull{10}{2}{0.1};
\ebull{12}{2}{0.1};
\ebull{14}{2}{0.1};
\bbull{16}{2}{0.1};
\ebull{18}{2}{0.1};
\bbull{2}{4}{0.1};
\ebull{4}{4}{0.1};
\bbull{6}{4}{0.1};
\ebull{8}{4}{0.1};
\ebull{10}{4}{0.1};
\ebull{12}{4}{0.1};
\bbull{14}{4}{0.1};
\bbull{16}{4}{0.1};
\ebull{18}{4}{0.1};
\ebull{3}{3}{0.1};
\bbull{5}{3}{0.1};
\ebull{7}{3}{0.1};
\ebull{9}{3}{0.1};
\bbull{11}{3}{0.1};
\bbull{13}{3}{0.1};
\ebull{15}{3}{0.1};
\ebull{17}{3}{0.1};
\node at (10,5) {$\xrightarrow[\hspace{40pt}]{\displaystyle \lambda}$};
\node at (10,1.2) {$\xrightarrow[\displaystyle \mu]{\hspace{40pt}}$};
\end{tikzpicture}
\]
Here we have restricted the picture so that the physical spaces to have indices in $[-4, 4]$, which is equivalent to considering $\ell(\lambda) \leq 4$ and $\lambda_1 \leq 4$ and $a_j = 0$ for all $j > 4$.
This finite lattice state has a Boltzmann weight of $(1+a_{-1} x_i) (1 + a_3 x_i) (1 + a_4 x_i) x_i^2$, and hence
\[
\bra{\lambda} \TMat(x) \ket{\mu} = (1+a_{-1} x_i) (1 + a_4 x_i) x_i^3 \prod_{j=5}^{\infty} (1 + a_j x_i).
\]
\end{ex}

As a consequence, for fixed partitions $\lambda, \mu$, if we restrict to a sufficiently large finite grid, we must have the physical space boundary be $0$ (resp.~$1$) for $\TMat$ (resp.~$\TMat^*$), or pictorially
\begin{equation}
\label{eq:finite_transfer_boundary}
\begin{tikzpicture}[scale=0.6,baseline=-1pt]
\draw[arrow=0.05,dscol] (-1,0)node[left,black]{$\ss 0$} -- (7,0) node[right,black] {$\ss 0$};
\foreach\x in {0,1,...,6}{
\draw[arrow=0.25] (\x,-1) -- (\x,1);
}
\node[below] at (0,-0.8) {\tiny $k_m$};\node[above] at (0,0.8) {\tiny $k'_m$};
\node[below] at (1.5,-0.8) {$\cdots$};\node[above] at (1.5,0.8) {$\cdots$};
\node[below] at (3,-0.8) {\tiny $k_0$};\node[above] at (3,0.8) {\tiny $k'_0$};
\node[below] at (6,-0.8) {\tiny $k_M$};\node[above] at (6,0.8) {\tiny $k'_M$};
\node[below] at (4.5,-0.8) {$\cdots$};\node[above] at (4.5,0.8) {$\cdots$};
\end{tikzpicture}
\qquad\qquad
\begin{tikzpicture}[scale=0.6,baseline=-1pt]
\draw[arrow=0.05,ddscol] (-1,0)node[left,black]{$\ss 1$} -- (7,0) node[right,black] {$\ss 1$};
\foreach\x in {0,1,...,6}{
\draw[arrow=0.25] (\x,-1) -- (\x,1);
}
\node[below] at (0,-0.8) {\tiny $k_m$};\node[above] at (0,0.8) {\tiny $k'_m$};
\node[below] at (1.5,-0.8) {$\cdots$};\node[above] at (1.5,0.8) {$\cdots$};
\node[below] at (3,-0.8) {\tiny $k_0$};\node[above] at (3,0.8) {\tiny $k'_0$};
\node[below] at (6,-0.8) {\tiny $k_M$};\node[above] at (6,0.8) {\tiny $k'_M$};
\node[below] at (4.5,-0.8) {$\cdots$};\node[above] at (4.5,0.8) {$\cdots$};
\end{tikzpicture}
\end{equation}
for $m \leq \ell(\lambda)$ and $M \geq \lambda_1$.

Next, we will specify the boundary conditions such that the resulting partition functions equal $\eschur^{\lambda/\mu}(\xx_n|\aaa)$.
Let $\fE_{\lambda/\mu}$ denote the vertex model using the $L$-matrix $L_{ij}$ on $[n] \times \ZZ$ with the top and bottom boundaries given by $\lambda$ and $\mu$ respectively.
Let $\fE_{\lambda/\mu}^*$ denote the corresponding dual model, with the top and bottom boundaries $\mu$ and $\lambda$, respectively, and using the $L^*$-matrix $L_{ij}^*$.

\begin{remark}
Unlike lattice models of $s_{\lambda}$ (such as $\aaa = 0$ for~\eqref{eq:dual_schur_vertices} or in Section~\eqref{sec:factorial_model} below), the partition function of the above model depends on size of the grid.
For a fixed $\lambda$, if one enlarges the grid from $n\times (n+\lambda_1)$ to $n\times (n+k)$ for some positive integer $k \geq \lambda_1$, then the partition function is scaled by $\prod_{j=\lambda_1}^k \prod_{i=1}^n (1 + a_j x_i)$.
We can see from the configuration in Example~\ref{ex:transfer_state} that only the deformed vertex can appear as a vertex in rightmost $n \times k$ grid.
\end{remark}

\begin{prop}
\label{prop:combinatorial_partition_function}
We have
\begin{align*}
\eschur^{\lambda/\mu}(\xx_n|\aaa) & = \bra{\lambda} \TMat(x_n) \dotsm \TMat(x_2) \TMat(x_1) \ket{\mu} = Z(\fE_{\lambda/\mu})
\\ & = \bra{\mu} \TMat^*(x_1) \TMat^*(x_2) \dotsm \TMat^*(x_n) \ket{\lambda} = Z(\fE^*_{\lambda/\mu}).
\end{align*}
\end{prop}

\begin{proof}
The proof that $\eschur^{\lambda/\mu}(\xx_n|\aaa) = Z(\fE_{\lambda/\mu})$ is similar to $\eschur^{\lambda/\mu}(\xx_n|\aaa) = Z(\fE^*_{\lambda/\mu})$, and so we will only do the first case.

We begin by noting the following \defn{branching rule}
\begin{equation}
\label{eq:branching_rule}
\eschur^{\lambda/\mu}(\xx_n, x_{n+1}|\aaa) = \sum_{\nu} \eschur^{\lambda/\nu}(x_{n+1}|\aaa) \eschur^{\nu/\mu}(\xx_n|\aaa)
\end{equation}
that follows immediately from the combinatorial definition by restricting to the labels into $[n]$ and $\{n+1\}$ in each edge labeled tableau.
Similarly, since $\sum_{\nu} \ket{\nu} \bra{\nu}$ is the identity matrix, we have
\[
\bra{\lambda} \TMat(x_{n+1}) \TMat(x_n) \dotsm \TMat(x_2) \TMat(x_1) \ket{\mu} = \sum_{\nu} \bra{\lambda} \TMat(x_{n+1}) \ket{\nu} \bra{\nu} \TMat(x_n) \dotsm \TMat(x_2) \TMat(x_1) \ket{\mu},
\]
which can also be seen pictorially by breaking the lattice model at the $n$-th row into two parts.
Hence, it is sufficient to show $\bra{\lambda} \TMat(x) \ket{\mu} = \eschur^{\lambda/\mu}(x|\aaa)$ for a single variable $x$.

We note that if $\lambda/\mu$ is not a horizontal strip,\footnote{A horizontal strip does not contain any vertical dominos, or equivalently it is a skew shape $\lambda / \mu$ that interlace $\lambda_1 \geq \mu_1 \geq \lambda_2 \geq \mu_2 \geq \cdots$.} then both sides are $0$.
Indeed, $\bra{\lambda} \TMat(x) \ket{\mu} = 0$ in this case as we can imagine the action of $\TMat(x)$ moves each particle to the right some number of steps without intersection.
Thus, assume $\lambda/\mu$ is a horizontal strip, and we have the equality as each step a particle moves to the right corresponds to a factor of $x$ and a box in the skew shape.
Each edge label above a box in the $j$-th diagonal corresponds to choosing the $a_j x_i$ factor in the Boltzmann weight $1 + a_j x_i$.
\end{proof}

\begin{ex}
We will translate the unique state of $\bra{\lambda} \TMat(x_i) \ket{\mu}$ from Example~\ref{ex:transfer_state} to edge labeled tableaux with $a_j = 0$ for all $j > 4$ or equivalently restricting to a finite grid.
Note that we can add an $i$ to the edges along the diagonals with index $-1$ and $4$, which are also the upper edges of boxes with those contents.
Indeed, we have the following edge labeled tableaux:
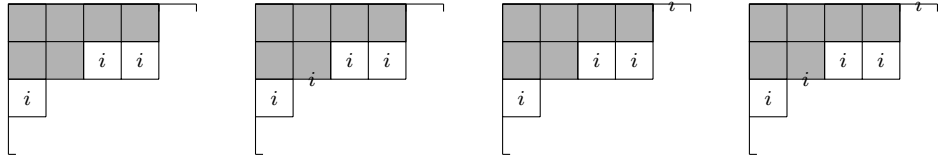
\[
\begin{array}{c@{\qquad}c@{\qquad}c@{\qquad}c}
\begin{tikzpicture}[scale=0.5]
\draw[fill=black!30] (0,0) rectangle (4,-1);
\draw[fill=black!30] (0,-1) rectangle (2,-2);
\draw (0,0) grid (1,-3);
\draw (2,0) grid (4,-2);
\draw (0,0) -- (5,0);
\draw(4,0)--(4,-0.2);
\draw(5,0)--(5,-0.2);
\draw(0,-3)--(0,-4);
\draw(0,-4)--(0.2,-4);
\node at (0.5,-2.5) {$\ss i$};
\node at (2.5,-1.5) {$\ss i$};
\node at (3.5,-1.5) {$\ss i$};
\end{tikzpicture}
&
\begin{tikzpicture}[scale=0.5]
\draw[fill=black!30] (0,0) rectangle (4,-1);
\draw[fill=black!30] (0,-1) rectangle (2,-2);
\draw (0,0) grid (1,-3);
\draw (2,0) grid (4,-2);
\draw (0,0) -- (5,0);
\draw(4,0)--(4,-0.2);
\draw(5,0)--(5,-0.2);
\draw(0,-3)--(0,-4);
\draw(0,-4)--(0.2,-4);
\node at (0.5,-2.5) {$\ss i$};
\node at (2.5,-1.5) {$\ss i$};
\node at (3.5,-1.5) {$\ss i$};
\node at (1.5,-2) {$\ss i$};
\end{tikzpicture}
&
\begin{tikzpicture}[scale=0.5]
\draw[fill=black!30] (0,0) rectangle (4,-1);
\draw[fill=black!30] (0,-1) rectangle (2,-2);
\draw (0,0) grid (1,-3);
\draw (2,0) grid (4,-2);
\draw (0,0) -- (5,0);
\draw(4,0)--(4,-0.2);
\draw(5,0)--(5,-0.2);
\draw(0,-3)--(0,-4);
\draw(0,-4)--(0.2,-4);
\node at (0.5,-2.5) {$\ss i$};
\node at (2.5,-1.5) {$\ss i$};
\node at (3.5,-1.5) {$\ss i$};
\node at (4.5,0) {$\ss i$};
\end{tikzpicture}
&
\begin{tikzpicture}[scale=0.5]
\draw[fill=black!30] (0,0) rectangle (4,-1);
\draw[fill=black!30] (0,-1) rectangle (2,-2);
\draw (0,0) grid (1,-3);
\draw (2,0) grid (4,-2);
\draw (0,0) -- (5,0);
\draw(4,0)--(4,-0.2);
\draw(5,0)--(5,-0.2);
\draw(0,-3)--(0,-4);
\draw(0,-4)--(0.2,-4);
\node at (0.5,-2.5) {$\ss i$};
\node at (2.5,-1.5) {$\ss i$};
\node at (3.5,-1.5) {$\ss i$};
\node at (1.5,-2) {$\ss i$};
\node at (4.5,0) {$\ss i$};
\end{tikzpicture}
\end{array}
\]
Note that we have the same partition function with the lattice $[1] \times [-k, 4]$ for any $k \geq 3$.
\end{ex}

\begin{ex}
\label{example:E(2,0)}
Let $\lambda = (2,0)$ and $a_j = 0$ for all $j > 2$.
We have the following lattice configurations for $\eschur^{\lambda}(\xx_2|\aaa)$:
\[
\begin{array}{r@{\;=\;}c@{\;\;+\;\;}c@{\;\;+\;\;}c}
\eschur^{\lambda}(\xx_2|\aaa) &
\begin{tikzpicture}[scale=0.6,baseline={([yshift=-1.5ex]current bounding box.center)}]
\node at (-1,3) {$\ss x_1$};
\node at (-1,4) {$\ss x_2$};
\foreach\x in {0,1}{
\draw[dscol] (0,3+\x) -- (5,3+\x);
\ebull{0}{3+\x}{0.1};
\ebull{5}{3+\x}{0.1};
}
\foreach\y in {1,2,3,4}{
\draw (\y,2) -- (\y,5);}
\node at (1,1.5) { $\ss z_{-2}$};
\node at (2,1.5) {$\ss z_{-1}$};
\node at (3,1.5) {$\ss z_{0}$};
\node at (4,1.5) {$\ss z_{1}$};
\bbull{1}{2}{0.1};
\bbull{2}{2}{0.1};
\ebull{3}{2}{0.1};
\ebull{4}{2}{0.1};
\bbull{1}{5}{0.1};
\ebull{2}{5}{0.1};
\ebull{3}{5}{0.1};
\bbull{4}{5}{0.1};
\ebull{1.5}{3}{0.1};
\bbull{2.5}{3}{0.1};
\bbull{3.5}{3}{0.1};
\ebull{1.5}{4}{0.1};
\ebull{2.5}{4}{0.1};
\ebull{3.5}{4}{0.1};
\bbull{1}{3.5}{0.1};
\ebull{2}{3.5}{0.1};
\ebull{3}{3.5}{0.1};
\bbull{4}{3.5}{0.1};
\node at (2.5,6) {$\xrightarrow[\hspace{40pt}]{\displaystyle \lambda}$};
\end{tikzpicture}
&
\begin{tikzpicture}[scale=0.6,baseline={([yshift=-1.5ex]current bounding box.center)}]
\node at (-1,3) {$\ss x_1$};
\node at (-1,4) {$\ss x_2$};
\foreach\x in {0,1}{
\draw[dscol] (0,3+\x) -- (5,3+\x);
\ebull{0}{3+\x}{0.1};
\ebull{5}{3+\x}{0.1};
}
\foreach\y in {1,2,3,4}{
\draw (\y,2) -- (\y,5);}
\node at (1,1.5) { $\ss z_{-2}$};
\node at (2,1.5) {$\ss z_{-1}$};
\node at (3,1.5) {$\ss z_{0}$};
\node at (4,1.5) {$\ss z_{1}$};
\bbull{1}{2}{0.1};
\bbull{2}{2}{0.1};
\ebull{3}{2}{0.1};
\ebull{4}{2}{0.1};
\bbull{1}{5}{0.1};
\ebull{2}{5}{0.1};
\ebull{3}{5}{0.1};
\bbull{4}{5}{0.1};
\ebull{1.5}{3}{0.1};
\bbull{2.5}{3}{0.1};
\ebull{3.5}{3}{0.1};
\ebull{1.5}{4}{0.1};
\ebull{2.5}{4}{0.1};
\bbull{3.5}{4}{0.1};
\bbull{1}{3.5}{0.1};
\ebull{2}{3.5}{0.1};
\bbull{3}{3.5}{0.1};
\ebull{4}{3.5}{0.1};
\node at (2.5,6) {$\xrightarrow[\hspace{40pt}]{\displaystyle \lambda}$};
\end{tikzpicture}
&
\begin{tikzpicture}[scale=0.6,baseline={([yshift=-1.5ex]current bounding box.center)}]
\node at (-1,3) {$\ss x_1$};
\node at (-1,4) {$\ss x_2$};
\foreach\x in {0,1}{
\draw[dscol] (0,3+\x) -- (5,3+\x);
\ebull{0}{3+\x}{0.1};
\ebull{5}{3+\x}{0.1};
}
\foreach\y in {1,2,3,4}{
\draw (\y,2) -- (\y,5);}
\node at (1,1.5) { $\ss z_{-2}$};
\node at (2,1.5) {$\ss z_{-1}$};
\node at (3,1.5) {$\ss z_{0}$};
\node at (4,1.5) {$\ss z_{1}$};
\bbull{1}{2}{0.1};
\bbull{2}{2}{0.1};
\ebull{3}{2}{0.1};
\ebull{4}{2}{0.1};
\bbull{1}{5}{0.1};
\ebull{2}{5}{0.1};
\ebull{3}{5}{0.1};
\bbull{4}{5}{0.1};
\ebull{1.5}{3}{0.1};
\ebull{2.5}{3}{0.1};
\ebull{3.5}{3}{0.1};
\ebull{1.5}{4}{0.1};
\bbull{2.5}{4}{0.1};
\bbull{3.5}{4}{0.1};
\bbull{1}{3.5}{0.1};
\bbull{2}{3.5}{0.1};
\ebull{3}{3.5}{0.1};
\ebull{4}{3.5}{0.1};
\node at (2.5,6) {$\xrightarrow[\hspace{40pt}]{\displaystyle \lambda}$};
\end{tikzpicture}
\\
& x^{2}_1 (1+a_{\ss -1} x_2)(1+a_{0} x_2)
& x_1 (1+a_{1} x_1)x_2 (1+a_{-1} x_2)
& x^{2}_2 (1+a_{0} x_1)(1+a_{1} x_1).
\end{array}
\]
The corresponding tableaux are
\begin{align*}
\ytableausetup{boxsize=1.5em}
\ytableaushort{11}
+
\ytableaushort{{\raisebox{-4pt}{$\substack{\displaystyle 1 \\[1pt] \displaystyle 2}$}}1}
+
\ytableaushort{1{\raisebox{-4pt}{$\substack{\displaystyle 1 \\[1pt] \displaystyle 2}$}}}
+
\ytableaushort{{\raisebox{-4pt}{$\substack{\displaystyle 1 \\[1pt] \displaystyle 2}$}}{\raisebox{-4pt}{$\substack{\displaystyle 1 \\[1pt] \displaystyle 2}$}}}
  & = x_1^{2}+a_{-1} x^{2}_1 x_2 +a_0 x^{2}_1 x_2+a_{-1} a_{0} x^{2}_1 x^{2}_2
  \\ & = x^{2}_1 (1+a_{-1} x_2)(1+a_{0} x_2),
\\[1em]
\ytableaushort{12}
+
\ytableaushort{{\raisebox{-4pt}{$\substack{\displaystyle 1 \\[1pt] \displaystyle 2}$}}2}
+
\ytableaushort{1{\raisebox{4pt}{$\substack{\displaystyle 1 \\[1pt] \displaystyle 2}$}}}
+
\ytableaushort{{\raisebox{-4pt}{$\substack{\displaystyle 1 \\[1pt] \displaystyle 2}$}}{\raisebox{4pt}{$\substack{\displaystyle 1 \\[1pt] \displaystyle 2}$}}}
  &=x_1 x_2+a_{-1} x_1 x^{2}_2+a_{1} x^{2}_1 x^{}_2+a_{-1} a_{1} x^{2}_1 x^{2}_2
  \\ &=x_1 (1+a_{1} x_1)x_2 (1+a_{-1} x_2),
\\[1em]
\ytableaushort{22}
+
\ytableaushort{{\raisebox{4pt}{$\substack{\displaystyle 1 \\[1pt] \displaystyle 2}$}}2}
+
\ytableaushort{2{\raisebox{4pt}{$\substack{\displaystyle 1 \\[1pt] \displaystyle 2}$}}} 
+
\ytableaushort{{\raisebox{4pt}{$\substack{\displaystyle 1 \\[1pt] \displaystyle 2}$}}{\raisebox{4pt}{$\substack{\displaystyle 1 \\[1pt] \displaystyle 2}$}}}
  &=x^{2}_2+a_0 x_1  x^{2}_2+a_1 x^{}_1 x^{2}_2+a_0 a_1 x^{2}_1 x^{2}_2
  \\ &=x^{2}_2 (1+a_0 x_1)(1+a_1 x_1),
\end{align*}
 \end{ex}

\begin{thm}
\label{thm:elambda-symmetric}
The function $\eschur^{\lambda/\mu}(\xx|\aaa)$ are invariant under permutation of $\xx$ variables.
Thus, $\eschur^{\lambda/\mu}(\xx|\aaa)$ is a symmetric function, and moreover $\{\eschur^{\lambda}(\xx|\aaa)\}_{\lambda}$ is a basis for symmetric functions.
\end{thm}

\begin{proof}
We will show that $\TMat(x) \TMat(y) = \TMat(y) \TMat(x)$, which implies the result.
We restrict to finite size transfer matrices, which by~\eqref{eq:finite_transfer_boundary} means we have the left and right boundary being holes:
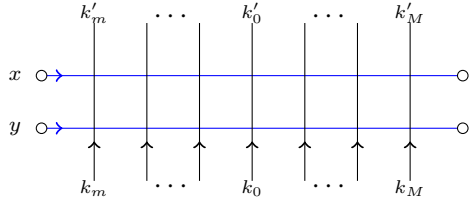
\[
\iftikz
\begin{tikzpicture}[scale=0.7]
\draw[arrow=0.05,dscol] (-1,0)node[left,black]{} -- (7,0) node[right,black] {};
\draw[arrow=0.05,dscol] (-1,-1)node[left,black]{} -- (7,-1) node[right,black] {};
\node at (-1.5,0) {$\ss x$};
\node at (-1.5,-1) {$\ss y$};
\foreach\x in {0,1,...,6}{
\draw[arrow=0.25] (\x,-2) -- (\x,1);
}
\ebull{-1}{0}{0.1};
\ebull{-1}{-1}{0.1};
\ebull{7}{0}{0.1};
\ebull{7}{-1}{0.1};
\node[below] at (0,-1.8) {\tiny $k_m$};\node[above] at (0,0.8) {\tiny $k'_m$};
\node[below] at (1.5,-1.8) {$\cdots$};\node[above] at (1.5,0.8) {$\cdots$};
\node[below] at (3,-1.8) {\tiny $k_0$};\node[above] at (3,0.8) {\tiny $k'_0$};
\node[below] at (6,-1.8) {\tiny $k_M$};\node[above] at (6,0.8) {\tiny $k'_M$};
\node[below] at (4.5,-1.8) {$\cdots$};\node[above] at (4.5,0.8) {$\cdots$};
\end{tikzpicture}
\fi
\]
Now multiply $\TMat(x)\TMat(y)$ on the right by $R(x/y)$ and then use the train argument, where repeatedly applying the Yang--Baxter equation yields
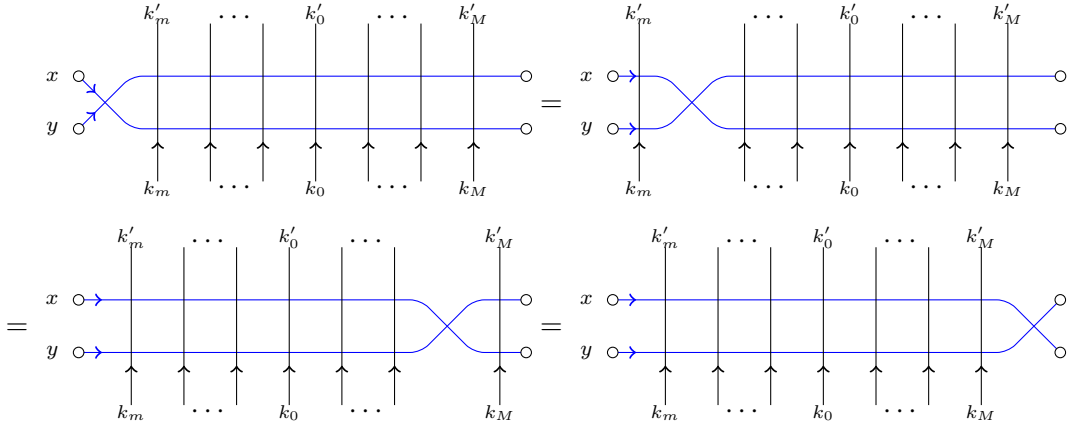
\iftikz
\begin{align*}
\begin{tikzpicture}[scale=0.7,baseline=-13pt]
\draw[dscol,rounded corners,arrow=0.05] (-1.5,0) -- (-0.5,-1) -- (7,-1);
\draw[dscol,rounded corners,arrow=0.05] (-1.5,-1) -- (-0.5,0) -- (7,0);
\ebull{-1.5}{0}{0.1};
\ebull{-1.5}{-1}{0.1};
\ebull{7}{0}{0.1};
\ebull{7}{-1}{0.1};
\node at (-2,0) {$\ss x$};
\node at (-2,-1) {$\ss y$};
\foreach\x in {0,1,...,6}{
\draw[arrow=0.25] (\x,-2) -- (\x,1);
}
\node[below] at (0,-1.8) {\tiny $k_m$};\node[above] at (0,0.8) {\tiny $k'_m$};
\node[below] at (1.5,-1.8) {$\cdots$};\node[above] at (1.5,0.8) {$\cdots$};
\node[below] at (3,-1.8) {\tiny $k_0$};\node[above] at (3,0.8) {\tiny $k'_0$};
\node[below] at (6,-1.8) {\tiny $k_M$};\node[above] at (6,0.8) {\tiny $k'_M$};
\node[below] at (4.5,-1.8) {$\cdots$};\node[above] at (4.5,0.8) {$\cdots$};
\end{tikzpicture}
& =
\begin{tikzpicture}[scale=0.7,baseline=-13pt]
\draw[dscol,rounded corners,arrow=0.05] (-1,0) -- (0,0) -- (1,-1) -- (7.5,-1);
\draw[dscol,rounded corners,arrow=0.05] (-1,-1) -- (0,-1) -- (1,0) -- (7.5,0);
\ebull{-1}{0}{0.1};
\ebull{-1}{-1}{0.1};
\ebull{7.5}{0}{0.1};
\ebull{7.5}{-1}{0.1};
\node at (-1.5,0) {$\ss x$};
\node at (-1.5,-1) {$\ss y$};
\foreach\x in {-0.5,1.5,2.5,...,6.5}{
\draw[arrow=0.25] (\x,-2) -- (\x,1);
}
\node[below] at (-0.5,-1.8) {\tiny $k_m$};\node[above] at (-0.5,0.8) {\tiny $k'_m$};
\node[below] at (2,-1.8) {$\cdots$};\node[above] at (2,0.8) {$\cdots$};
\node[below] at (3.5,-1.8) {\tiny $k_0$};\node[above] at (3.5,0.8) {\tiny $k'_0$};
\node[below] at (6.5,-1.8) {\tiny $k_M$};\node[above] at (6.5,0.8) {\tiny $k'_M$};
\node[below] at (5,-1.8) {$\cdots$};\node[above] at (5,0.8) {$\cdots$};
\end{tikzpicture}
\allowdisplaybreaks \\
= 
\begin{tikzpicture}[scale=0.7,baseline=-13pt]
\draw[dscol,rounded corners,arrow=0.05] (-1,0) -- (5.5,0) -- (6.5,-1) -- (7.5,-1);
\draw[dscol,rounded corners,arrow=0.05] (-1,-1) -- (5.5,-1) -- (6.5,0) -- (7.5,0);
\ebull{-1}{0}{0.1};
\ebull{-1}{-1}{0.1};
\ebull{7.5}{0}{0.1};
\ebull{7.5}{-1}{0.1};
\node at (-1.5,0) {$\ss x$};
\node at (-1.5,-1) {$\ss y$};
\foreach\x in {0,1,...,5,7}{
\draw[arrow=0.25] (\x,-2) -- (\x,1);
}
\node[below] at (0,-1.8) {\tiny $k_m$};\node[above] at (0,0.8) {\tiny $k'_m$};
\node[below] at (1.5,-1.8) {$\cdots$};\node[above] at (1.5,0.8) {$\cdots$};
\node[below] at (3,-1.8) {\tiny $k_0$};\node[above] at (3,0.8) {\tiny $k'_0$};
\node[below] at (7,-1.8) {\tiny $k_M$};\node[above] at (7,0.8) {\tiny $k'_M$};
\node[below] at (4.5,-1.8) {$\cdots$};\node[above] at (4.5,0.8) {$\cdots$};
\end{tikzpicture}
& =
\begin{tikzpicture}[scale=0.7,baseline=-13pt]
\draw[dscol,rounded corners,arrow=0.05] (-1.5,0) -- (6,0) -- (7,-1);
\draw[dscol,rounded corners,arrow=0.05] (-1.5,-1) -- (6,-1) -- (7,0);
\ebull{-1.5}{0}{0.1};
\ebull{-1.5}{-1}{0.1};
\ebull{7}{0}{0.1};
\ebull{7}{-1}{0.1};
\node at (-2,0) {$\ss x$};
\node at (-2,-1) {$\ss y$};
\foreach\x in {0,1,...,6}{
\draw[arrow=0.25] (\x-.5,-2) -- (\x-.5,1);
}
\node[below] at (-.5,-1.8) {\tiny $k_m$};\node[above] at (-0.5,0.8) {\tiny $k'_m$};
\node[below] at (1,-1.8) {$\cdots$};\node[above] at (1,0.8) {$\cdots$};
\node[below] at (2.5,-1.8) {\tiny $k_0$};\node[above] at (2.5,0.8) {\tiny $k'_0$};
\node[below] at (5.5,-1.8) {\tiny $k_M$};\node[above] at (5.5,0.8) {\tiny $k'_M$};
\node[below] at (4,-1.8) {$\cdots$};\node[above] at (4,0.8) {$\cdots$};
\end{tikzpicture}
\end{align*}
\fi
We note that on both sides there is only one possible non-zero entry of the $R$-matrix for the corresponding boundary conditions and the weight of the entry is $1$.
Thus, we have that $\TMat(y) \TMat(x) \cdot 1 = 1 \cdot \TMat(x) \TMat(y)$ on the finite size transfer matrices, which extends to the full transfer matrices.
\end{proof}

We remark that a similar proof can be given using $\TMat^*$.

\subsection{Lattice model for factorial Schur polynomials}
\label{sec:factorial_model}

In this subsection, we describe a lattice model whose partition function is $s_{\lambda/\mu}(\xx|\aaa)$.
This model is originally due to Zinn-Justin~\cite{ZJ09}.
Note that when $\aaa = 0$, this is a classical five-vertex model for the Schur functions.

Define the Boltzmann weights for the $\ell$-matrix as
\begin{gather*}
\begin{array}{c*{4}{@{\hspace{40pt}}c}}
\textsf{a}_1 & \textsf{a}_2 & \textsf{b}_2 & \textsf{c}_1 & \textsf{c}_2
\\
\begin{tikzpicture}[scale=0.4,baseline=-2pt]
\draw[arrow=0.25,scol] (-1,0) node[left,black] {$\ss {0}$} -- (1,0) node[right,black] {$\ss {0}$};
\draw[arrow=0.25] (0,-1) node[below] {$\ss {0}$} -- (0,1) node[above] {$\ss {0}$};
\end{tikzpicture}
&
\begin{tikzpicture}[scale=0.4,baseline=-2pt]
\draw[arrow=0.25,scol] (-1,0) node[left,black] {$\ss {1}$} -- (1,0) node[right,black] {$\ss {1}$};
\draw[arrow=0.25] (0,-1) node[below] {$\ss {1}$} -- (0,1) node[above] {$\ss {1}$};
\end{tikzpicture}
&
\begin{tikzpicture}[scale=0.4,baseline=-2pt]
\draw[arrow=0.25,scol] (-1,0) node[left,black] {$\ss {1}$} -- (1,0) node[right,black] {$\ss {1}$};
\draw[arrow=0.25] (0,-1) node[below] {$\ss {0}$} -- (0,1) node[above] {$\ss{0}$};
\end{tikzpicture}
&
\begin{tikzpicture}[scale=0.4,baseline=-2pt]
\draw[arrow=0.25,scol] (-1,0) node[left,black] {$\ss {0}$} -- (1,0) node[right,black] {$\ss {1}$};
\draw[arrow=0.25] (0,-1) node[below] {$\ss {1}$} -- (0,1) node[above] {$\ss {0}$};
\end{tikzpicture}
&
\begin{tikzpicture}[scale=0.4,baseline=-2pt]
\draw[arrow=0.25,scol] (-1,0) node[left,black] {$\ss {1}$} -- (1,0) node[right,black] {$\ss {0}$};
\draw[arrow=0.25] (0,-1) node[below] {$\ss {0}$} -- (0,1) node[above] {$\ss {1}$};
\end{tikzpicture}
\\
1 & 1 & x_i - a_j & 1 & 1
\end{array}
\allowdisplaybreaks
\\
\ell_{ij}(x_i,a_j) =
\begin{pmatrix}
\begin{tikzpicture}[scale=0.6,baseline=4pt,rotate=45]
\draw (0,0)--(1,1);
\draw[scol] (0,1)--(1,0);
\ebull{0}{1}{0.1};
\ebull{1}{0}{0.1};
\ebull{1}{1}{0.1};
\ebull{0}{0}{0.1};
\end{tikzpicture}& 0 & 0 & 0 \\[1ex]
0 & \begin{tikzpicture}[scale=0.6,baseline=4pt,rotate=45]
\draw (0,0)--(1,1);
\draw[scol] (0,1)--(1,0);
\bbull{0}{0}{0.1};
\ebull{0}{1}{0.1};
\ebull{1}{0}{0.1};
\bbull{1}{1}{0.1};
\end{tikzpicture} & \begin{tikzpicture}[scale=0.6,baseline=4pt,rotate=45]
\draw (0,0)--(1,1);
\draw[scol] (0,1)--(1,0);
\ebull{0}{1}{0.1};
\ebull{1}{1}{0.1};
\bbull{1}{0}{0.1};
\bbull{0}{0}{0.1};
\end{tikzpicture} & 0 \\[1ex]
0 &\begin{tikzpicture}[scale=0.6,baseline=4pt,rotate=45]
\draw (0,0)--(1,1);
\draw[scol] (0,1)--(1,0);
\bbull{0}{1}{0.1};
\ebull{0}{0}{0.1};
\ebull{1}{0}{0.1};
\bbull{1}{1}{0.1};
\end{tikzpicture} & \begin{tikzpicture}[scale=0.6,baseline=4pt,rotate=45]
\draw (0,0)--(1,1);
\draw[scol] (0,1)--(1,0);
\bbull{1}{0}{0.1};
\bbull{0}{1}{0.1};
\ebull{1}{1}{0.1};
\ebull{0}{0}{0.1};
\end{tikzpicture} & 0 \\[1ex]
0 & 0 & 0 & \begin{tikzpicture}[scale=0.6,baseline=4pt,rotate=45]
\draw (0,0)--(1,1);
\draw[scol] (0,1)--(1,0);
\bbull{1}{0}{0.1};
\bbull{0}{1}{0.1};
\bbull{1}{1}{0.1};
\bbull{0}{0}{0.1};
\end{tikzpicture} \\
\end{pmatrix}_{ij}
=
\begin{pmatrix}
1 & 0 & 0 & 0 \\[1ex]
0 & 0 & 1 & 0 \\[1ex]
0 & 1 & x_{i}-a_{j} & 0 \\[1ex]
0 & 0 & 0 & 1 \\
\end{pmatrix}_{ij}
\in
\End(H_i \otimes V_j).
\end{gather*}

Similar to the previous section, we want to build a single row transfer matrix $\tmat(x_i)$ and then build the lattice model by a product of transfer matrices.
As before, define
\[
\tmat(x_i) = \prod_{j \in \ZZ}L_{ij}(x_i,a_j)
\qquad
\left(
\begin{tikzpicture}[scale=0.6,baseline=-1pt]
\draw[arrow=0.05,scol] (-1,0)node[left,black]{} -- (7,0) node[right,black] {};
\foreach\x in {0,1,...,6}{
\draw[arrow=0.25] (\x,-1) -- (\x,1);
}
\node[below] at (0,-0.8) {$\cdots$};\node[above] at (0,0.8) {$\cdots$};
\node[below] at (1,-0.8) {\tiny $k_{-2}$};\node[above] at (1,0.8) {\tiny $k'_{-2}$};
\node[below] at (2,-0.8) {\tiny $k_{-1}$};\node[above] at (2,0.8) {\tiny $k'_{-1}$};
\node[below] at (3,-0.8) {\tiny $k_0$};\node[above] at (3,0.8) {\tiny $k'_0$};
\node[below] at (4,-0.8) {\tiny $k_1$};\node[above] at (4,0.8) {\tiny $k'_1$};
\node[below] at (5,-0.8) {\tiny $k_2$};\node[above] at (5,0.8) {\tiny $k'_2$};
\node[below] at (6,-0.8) {$\cdots$}; \node[above] at (6,0.8) {$\cdots$};
\end{tikzpicture}
\right).
\]
Due to our $\ell$-matrix, we need to impose a technical condition that $\sup_j \abs{x_i - a_j} < 1$ (and a fixed $i$).
This means that there cannot be an infinite products, or for any sufficiently large finite transfer matrix, we cannot have the right boundary condition being $1$ (a particle).
In terms of an infinite transfer matrix describing particle motion, it means that no particle can ``escape'' to $\infty$.
Pictorially, for a finite transfer matrix, we impose the boundary conditions
\[
\begin{tikzpicture}[scale=0.6,baseline=-1pt]
\draw[arrow=0.05,scol] (-1,0)node[left,black]{$\ss 1$} -- (7,0) node[right,black] {$\ss 0$};
\foreach\x in {0,1,...,6}{
\draw[arrow=0.25] (\x,-1) -- (\x,1);
}
\node[below] at (0,-0.8) {\tiny $k_m$};\node[above] at (0,0.8) {\tiny $k'_m$};
\node[below] at (1.5,-0.8) {$\cdots$};\node[above] at (1.5,0.8) {$\cdots$};
\node[below] at (3,-0.8) {\tiny $k_0$};\node[above] at (3,0.8) {\tiny $k'_0$};
\node[below] at (6,-0.8) {\tiny $k_M$};\node[above] at (6,0.8) {\tiny $k'_M$};
\node[below] at (4.5,-0.8) {$\cdots$};\node[above] at (4.5,0.8) {$\cdots$};
\end{tikzpicture}
\]

Next, we note that every particle must move at least one position in $\tmat$.
Therefore, we introduce new notation to account for this shift by $\ket{\lambda}_{[\kappa]} = \ket{\lambda + \delta_{\kappa}}$, where the partition $\lambda + \delta_{\kappa} = (\lambda_i + \kappa - i)_{i=1}^{\kappa}$.
Here, our partition function is independent of the width (provided it is sufficiently wide), and so we can work in the finite grid $[n] \times [\lambda_1]$.
We impose the boundary conditions of the left and right boundaries being $1$ and $0$ respectively.
Similar to before, the top and bottom boundaries will be $\lambda + \delta_{k+n}$ and $\mu +\delta_k$.

\begin{thm}[{\cite{ZJ09}}]
\label{thm:factorial_partition_func}
For any $\kappa \geq \ell(\mu)$, we have
\[
s_{\lambda/\mu}(\xx_n | \aaa) = {}_{[\kappa+n]}\! \bra{\lambda} \tmat(x_n) \dotsm \tmat(x_2) \tmat(x_1) \ket{\mu}_{[\kappa]}.
\]
Furthermore, this model is integrable with the $r$-matrix given by
\[
 r_{ij}(x_i, x_j)=
 \begin{pmatrix}
\begin{tikzpicture}[scale=0.6,baseline=4pt]
\draw[scol] (0,0)--(1,1);
\draw[scol] (0,1)--(1,0);
\ebull{0}{1}{0.1};
\ebull{1}{0}{0.1};
\ebull{1}{1}{0.1};
\ebull{0}{0}{0.1};
\end{tikzpicture}& 0 & 0 & 0 \\[1ex]
0 & \begin{tikzpicture}[scale=0.6,baseline=4pt]
\draw[scol] (0,0)--(1,1);
\draw[scol] (0,1)--(1,0);
\bbull{0}{0}{0.1};
\ebull{0}{1}{0.1};
\ebull{1}{0}{0.1};
\bbull{1}{1}{0.1};
\end{tikzpicture} & \begin{tikzpicture}[scale=0.6,baseline=4pt]
\draw[scol] (0,0)--(1,1);
\draw[scol] (0,1)--(1,0);
\ebull{0}{1}{0.1};
\ebull{1}{1}{0.1};
\bbull{1}{0}{0.1};
\bbull{0}{0}{0.1};
\end{tikzpicture} & 0 \\[1ex]
0 &\begin{tikzpicture}[scale=0.6,baseline=4pt]
\draw[scol] (0,0)--(1,1);
\draw[scol] (0,1)--(1,0);
\bbull{0}{1}{0.1};
\ebull{0}{0}{0.1};
\ebull{1}{0}{0.1};
\bbull{1}{1}{0.1};
\end{tikzpicture} & \begin{tikzpicture}[scale=0.6,baseline=4pt]
\draw[scol] (0,0)--(1,1);
\draw[scol] (0,1)--(1,0);
\bbull{1}{0}{0.1};
\bbull{0}{1}{0.1};
\ebull{1}{1}{0.1};
\ebull{0}{0}{0.1};
\end{tikzpicture} & 0 \\[1ex]
0 & 0 & 0 & \begin{tikzpicture}[scale=0.6,baseline=4pt]
\draw[scol] (0,0)--(1,1);
\draw[scol] (0,1)--(1,0);
\bbull{1}{0}{0.1};
\bbull{0}{1}{0.1};
\bbull{1}{1}{0.1};
\bbull{0}{0}{0.1};
\end{tikzpicture} \\
\end{pmatrix}_{ij}
=
\begin{pmatrix}
1 & 0 & 0 & 0 \\[1ex]
0 & 0 & 1 & 0 \\[1ex]
0 & 1 & x_i - x_j & 0 \\[1ex]
0 & 0 & 0 & 1 \\
\end{pmatrix}_{ij}
\in 
\End(H_i \otimes H_j).
\]
\end{thm}

\begin{remark}
This vertex model model is related to the model of~\cite{BMN14} at $t = 0$ by flipping the horizontal edges and scaling the vertices $\mathsf{a}_2$ and $\mathsf{c}_1$ by $x_i$.
The effect on our partition function is multiplying by $x_1^1 \cdots x_n^n$ as there are exactly $i$ such vertices in the $i$-th row (contrast Theorem~\ref{thm:factorial_partition_func} with~\cite[Thm.~1]{BMN14}).
This answers the question poised about the relationship between the models poised in~\cite{BMN14}.
\end{remark}

For the remainder of the paper, we will take $\kappa = 0$.
We can build a semistandard tableau $T$ for each state in the lattice model by having the $i$-th row ${}_{[i]}\! \bra{\lambda^{(i)}}$ correspond to the shape of $T$ restricted to the entries at most $i$.
Thus, the sequence $(\lambda^{(1)}, \dotsc, \lambda^{(n)})$ forms the Gelfand--Tsetlin pattern corresponding to $T$.

\begin{ex}
Let $\lambda = (2,0)$.
We have the following lattice configurations for $s^{\lambda}(\xx_2|\aaa)$:
\[
\begin{array}{r@{\;=\;}c@{\;\;+\;\;}c@{\;\;+\;\;}c}
s_{(2,0)}(\xx_1|\aaa) &
\begin{tikzpicture}[scale=0.6,baseline={([yshift=-1.5ex]current bounding box.center)}]
\foreach\x in {0,1}{
\draw[scol] (0,3+\x) -- (5,3+\x);
\bbull{0}{3+\x}{0.1};
\ebull{5}{3+\x}{0.1};
}
\node at (-0.5,3) {$\ss x_{1}$};
\node at (-0.5,4) {$\ss x_{2}$};
\foreach\y in {1,2,3,4}{
\draw (\y,2)-- (\y,5);
 \node at (\y,1.5){$\ss a_\y$};}
\ebull{1}{2}{0.1};
\ebull{2}{2}{0.1};
\ebull{3}{2}{0.1};
\ebull{4}{2}{0.1};
\bbull{1}{5}{0.1};
\ebull{2}{5}{0.1};
\ebull{3}{5}{0.1};
\bbull{4}{5}{0.1};
\bbull{1.5}{3}{0.1};
\bbull{2.5}{3}{0.1};
\ebull{3.5}{3}{0.1};
\ebull{1.5}{4}{0.1};
\ebull{2.5}{4}{0.1};
\bbull{3.5}{4}{0.1};
\ebull{1}{3.5}{0.1};
\ebull{2}{3.5}{0.1};
\bbull{3}{3.5}{0.1};
\ebull{4}{3.5}{0.1};
\node at (2.5,6) {$\xrightarrow[\hspace{40pt}]{\displaystyle \lambda}$};
\end{tikzpicture}
&
\begin{tikzpicture}[scale=0.6,baseline={([yshift=-1.5ex]current bounding box.center)}]
\foreach\x in {0,1}{
\draw[scol] (0,3+\x) -- (5,3+\x);
\bbull{0}{3+\x}{0.1};
\ebull{5}{3+\x}{0.1};
}
\foreach\y in {1,2,3,4}{
\draw (\y,2) -- (\y,5);
 \node at (\y,1.5){$\ss a_\y$};}
\ebull{1}{2}{0.1};
\ebull{2}{2}{0.1};
\ebull{3}{2}{0.1};
\ebull{4}{2}{0.1};
\bbull{1}{5}{0.1};
\ebull{2}{5}{0.1};
\ebull{3}{5}{0.1};
\bbull{4}{5}{0.1};
\bbull{1.5}{3}{0.1};
\ebull{2.5}{3}{0.1};
\ebull{3.5}{3}{0.1};
\ebull{1.5}{4}{0.1};
\bbull{2.5}{4}{0.1};
\bbull{3.5}{4}{0.1};
\ebull{1}{3.5}{0.1};
\bbull{2}{3.5}{0.1};
\ebull{3}{3.5}{0.1};
\ebull{4}{3.5}{0.1};
\node at (2.5,6) {$\xrightarrow[\hspace{40pt}]{\displaystyle \lambda}$};
\end{tikzpicture}
&
\begin{tikzpicture}[scale=0.6,baseline={([yshift=-1.5ex]current bounding box.center)}]
\foreach\x in {0,1}{
\draw[scol] (0,3+\x) -- (5,3+\x);
\bbull{0}{3+\x}{0.1};
\ebull{5}{3+\x}{0.1};
}
\foreach\y in {1,2,3,4}{
\draw (\y,2) -- (\y,5);
\node at (\y,1.5){$\ss a_\y$};}
\ebull{1}{2}{0.1};
\ebull{2}{2}{0.1};
\ebull{3}{2}{0.1};
\ebull{4}{2}{0.1};
\bbull{1}{5}{0.1};
\ebull{2}{5}{0.1};
\ebull{3}{5}{0.1};
\bbull{4}{5}{0.1};
\ebull{1.5}{3}{0.1};
\ebull{2.5}{3}{0.1};
\ebull{3.5}{3}{0.1};
\bbull{1.5}{4}{0.1};
\bbull{2.5}{4}{0.1};
\bbull{3.5}{4}{0.1};
\bbull{1}{3.5}{0.1};
\ebull{2}{3.5}{0.1};
\ebull{3}{3.5}{0.1};
\ebull{4}{3.5}{0.1};
\node at (2.5,6) {$\xrightarrow[\hspace{40pt}]{\displaystyle \lambda}$};
\end{tikzpicture}
\\
& (x_1 - a_1)(x_1 - a_2)
& (x_1 - a_1)(x_2 - a_3)
& (x_2 - a_2)(x_2 - a_3),
\end{array}
\]
where the corresponding semistandard tableaux are
\[
\ytableaushort{11}\,,
\qquad\qquad
\ytableaushort{12}\,,
\qquad\qquad
\ytableaushort{22}\,.
\]
\end{ex}

\subsection{Cauchy identity}
\label{sec:cauchyidentity}

In this section, we prove our Cauchy type identity by deriving a relation between matrices $\TMat^*$ and $\tmat$.
For the remainder of this section, we will replace $\aaa$ with $-\aaa$ in the $\ell$-matrix and transfer matrix $\tmat$.
We will also assume the quantum space for $\TMat^*$ will use the dual space $H^*$ to distinguish it (this is also natural from the fact we want to reverse the order of the pairing for computing the partition function).

\begin{lemma}
\label{lemma:commcauchy}
We have the following commutation relation:
\[
\TMat^*(y) \tmat(x) = \frac{1}{1 - x y} \tmat(x) \TMat^*(y).
\]
\end{lemma}

\begin{proof}
The proof is similar to the way we proved that the transfer matrices commute for Theorem~\ref{thm:elambda-symmetric}.
The $\mathfrak{R}$-matrix
\[
\mathfrak{R}_{ij}(x,y)=
\begin{pmatrix}
\begin{tikzpicture}[scale=0.6,baseline=4pt]
\draw[scol] (0,0)--(1,1);
\draw[ddscol] (0,1)--(1,0);
\ebull{0}{1}{0.1};
\ebull{1}{0}{0.1};
\ebull{1}{1}{0.1};
\ebull{0}{0}{0.1};
\end{tikzpicture}& 0 & 0 & 0 \\[1ex]
0 & \begin{tikzpicture}[scale=0.6,baseline=4pt]
\draw[scol] (0,0)--(1,1);
\draw[ddscol] (0,1)--(1,0);
\bbull{0}{0}{0.1};
\ebull{0}{1}{0.1};
\ebull{1}{0}{0.1};
\bbull{1}{1}{0.1};
\end{tikzpicture} & \begin{tikzpicture}[scale=0.6,baseline=4pt]
\draw[scol] (0,0)--(1,1);
\draw[ddscol] (0,1)--(1,0);
\ebull{0}{1}{0.1};
\ebull{1}{1}{0.1};
\bbull{1}{0}{0.1};
\bbull{0}{0}{0.1};
\end{tikzpicture} & 0 \\[1ex]
0 &\begin{tikzpicture}[scale=0.6,baseline=4pt]
\draw[scol] (0,0)--(1,1);
\draw[ddscol] (0,1)--(1,0);
\bbull{0}{1}{0.1};
\ebull{0}{0}{0.1};
\ebull{1}{0}{0.1};
\bbull{1}{1}{0.1};
\end{tikzpicture} & \begin{tikzpicture}[scale=0.6,baseline=4pt]
\draw[scol] (0,0)--(1,1);
\draw[ddscol] (0,1)--(1,0);
\bbull{1}{0}{0.1};
\bbull{0}{1}{0.1};
\ebull{1}{1}{0.1};
\ebull{0}{0}{0.1};
\end{tikzpicture} & 0 \\[1ex]
0 & 0 & 0 & \begin{tikzpicture}[scale=0.6,baseline=4pt,]
\draw[scol] (0,0)--(1,1);
\draw[ddscol] (0,1)--(1,0);
\bbull{1}{0}{0.1};
\bbull{0}{1}{0.1};
\bbull{1}{1}{0.1};
\bbull{0}{0}{0.1};
\end{tikzpicture} \\
\end{pmatrix}_{ij}
=
\begin{pmatrix}
y & 0 & 0 & 0 \\[1ex]
0 & 0 & y & 0 \\[1ex]
0 & 1 & 1-x y & 0 \\[1ex]
0 & 0 & 0 & 1 \\
\end{pmatrix}_{ij}
\in 
\End(H \otimes H^*),
\]
together with $L$-matrix and $\ell$-matrix satisfy the $\mathfrak{R}L\ell$ relation:
\[
\ell_{jk}(x, a_k) L_{ik}(y, a_k) \mathfrak{R}_{ij}(x,y) = \mathfrak{R}_{ij}(x,y) L_{ik}(y, a_k) \ell_{jk}(x, a_k)
\quad
\left(
\begin{tikzpicture}[scale=0.6,baseline=7pt]
\draw(1,-1)--(1,2);
\draw[rounded corners,arrow=0.15,ddscol] (-1,1) node[left,black]{$\ss y$}--(0,0)--(2,0);
\draw[rounded corners,arrow=0.15,scol] (-1,0)  node[left,black]{$\ss x$}--(0,1)--(2,1);
\end{tikzpicture}
=
\begin{tikzpicture}[scale=0.6,baseline=7pt]
\draw (1,2)--(1,-1);
\draw[arrow=0.15,rounded corners,scol] (0,0) node[left,black]{$\ss x$}--(2,0)--(3,1);
\draw[arrow=0.15,rounded corners,ddscol] (0,1)node[left,black]{$\ss y$}--(2,1)--(3,0);
\end{tikzpicture}
\right).
\]

After multiplying the $\mathfrak{R}$-matrix to $\tmat(x) \TMat^*(y)$ and repeated application of the $\mathfrak{R}L\ell$ relation we obtain the equation
\[
\iftikz
\begin{tikzpicture}[scale=0.6,baseline=3pt]
\draw[ddscol,rounded corners,arrow=0.15] (-1,1)node[black,left]{$\ss 1$}--(0,0) node[below,black] {}--(7,0) node[black,right]{$\ss 1$};
\draw[scol,rounded corners,arrow=0.15](-1,0)node[black,left]{$\ss 1$}-- (0,1)node[above,black] {}--(7,1)  node[black,right]{$\ss 0$};
\foreach \x in {1,2,3,4,5}{
\draw (\x,-1)--(\x,2);
\node at (\x,2.2) {$\ss i_{\x}$};
\node at (\x,-1.2) {$\ss k_{\x}$};
}
\draw (6,-1)--(6,2);
\node at (6,2.2) {$\ss \dots$};
\node at (6,-1.2) {$\ss \dots$};
\node at (-2,1) {$\ss y$};
\node at (-2,0) {$\ss x$};
\end{tikzpicture}
=
\begin{tikzpicture}[scale=0.6,baseline=3pt]
\draw[scol,rounded corners,arrow=0.15] (0,0) node[left,black]{$\ss 1$}--(7,0) node[below,black]{}--(8,1) node[right,black]{$\ss 0$};
\draw[ddscol,rounded corners,arrow=0.15] (0,1)node[left,black]{$\ss 1$}--(7,1)node[above,black]{}--(8,0) node[right,black]{$\ss 1$};
\foreach \x in {1,2,3,4,5}{
\draw (\x,-1)--(\x,2);
\node at (\x,2.2)  {$\ss i_{\x}$};
\node at (\x,-1.2) {$\ss k_{\x}$};
}
\draw (6,-1)--(6,2);
\node at (6,2.2)  {$\ss \dots$};
\node at (6,-1.2) {$\ss \dots$};
\node at (-1,0) {$\ss x$};
\node at (-1,1) {$\ss y$};
\end{tikzpicture}
\fi
\]
On the left side, as there is a unique entry for the $\mathfrak{R}$-matrix with Boltzmann weight $1$, the partition function is $\tmat(x) \TMat^*(y)$.
On the right side there are two possible configurations:
\[
\iftikz
\begin{tikzpicture}[scale=0.6,baseline=3pt]
\draw[scol,rounded corners] (0,0) node[left,black]{$\ss 1$}--(7,0) node[below,black]{$\ss 1$}--(8,1) node[right,black]{$\ss 0$};
\draw[ddscol,rounded corners] (0,1)node[left,black]{$\ss 1$}--(7,1)node[above,black]{$\ss 0$}--(8,0) node[right,black]{$\ss 1$};
\foreach \x in {1,2,3,4,5}{
\draw (\x,-1)--(\x,2);
\node at (\x,2.2)  {$\ss i_{\x}$};
\node at (\x,-1.2) {$\ss k_{\x}$};
}
\draw (6,-1)--(6,2);
\node at (6,2.2)  {$\ss \dots$};
\node at (6,-1.2) {$\ss \dots$};
\node at (-1,0) {$\ss x$};
\node at (-1,1) {$\ss y$};
\end{tikzpicture}
+
\begin{tikzpicture}[scale=0.6,baseline=3pt]
\draw[scol,rounded corners] (0,0) node[left,black]{$\ss 1$}--(7,0) node[below,black]{$\ss 0$}--(8,1) node[right,black]{$\ss 0$};
\draw[ddscol,rounded corners] (0,1)node[left,black]{$\ss 1$}--(7,1)node[above,black]{$\ss 1$}--(8,0) node[right,black]{$\ss 1$};
\foreach \x in {1,2,3,4,5}{
\draw (\x,-1)--(\x,2);
\node at (\x,2.2)  {$\ss i_{\x}$};
\node at (\x,-1.2) {$\ss k_{\x}$};
}
\draw(6,-1)--(6,2);
\node at (6,2.2)  {$\ss \dots$};
\node at (6,-1.2) {$\ss \dots$};
\node at (-1,0) {$\ss x$};
\node at (-1,1) {$\ss y$};
\end{tikzpicture}
\fi
\]
Recall that we assume that $\sup_k \abs{x - a_k} < 1$, and therefore the partition function of the first configuration is $0$ (as otherwise it would mean a particle could ``escape'' to $\infty$ for the bottom row).
For the second configuration, the Boltzmann weight of the $\mathfrak{R}$-matrix is $1 - xy$.
This means we have the commutation relation
\[
\tmat(x) \TMat^*(y) = (1-xy) \TMat^*(y) \tmat(x)
\]
as desired.
\end{proof}

We can now state the identity that relates $s_{\lambda}$ and $\eschur^{\lambda}$.

\begin{thm}[Skew Cauchy identity]
\label{thm:cauchy}
For a positive integer $n$, define $\aaa^{n} = (a^{n}_{i})_{i \in \ZZ_{>0}}$, where $a^{n}_{i} = a_{i-n-1}$.
Assume $a^n_i = 0$ for all $i > N$ for some $N \in \ZZ_{>0} \sqcup \{\infty\}$.
The factorial Schur polynomials and edge Schur functions satisfy the relation
\[
\sum_{\lambda} s_{\lambda/\mu}(\xx_n | {-\aaa}) \eschur^{\lambda/\eta}(\yy_m | \aaa^{n}) = \prod_{\substack{1\leq i\leq n\\[0.2em]1\leq j\leq m}}\frac{1}{1-x_i y_j} \sum_{\kappa} s_{\eta/\kappa}(\xx_n | {-\aaa}) E^{\mu/\kappa}(\yy_m|\aaa^n).
\]
\end{thm}

\begin{proof}
The proof is essentially using the Yang--Baxter relation to relate
\begin{align*}
{}_n\!\!\,\bra{\eta} \TMat^*(y_1) \dotsm \TMat^*(y_m) \tmat(x_n) \dotsm \tmat(x_1) \ket{\mu} & = \sum_{\lambda} s_{\lambda/\mu}(\xx_n|{-\aaa}) \eschur^{\lambda/\eta}(\yy_m|\aaa^n),
\\
{}_n\!\!\,\bra{\eta} \tmat(x_1) \dotsm \tmat(x_n) \TMat^*(y_m) \dotsm \TMat^*(y_1) \ket{\mu} & = \sum_{\kappa} s_{\eta/\kappa}(\xx_n | {-\aaa}) E^{\mu/\kappa}(\yy_m|\aaa^n),
\end{align*}
where ${}_n\!\!\,\bra{\eta}$ denotes $\bra{\eta}$ but with each factor $\bra{k_i}_i \mapsto \bra{k_i}_{i+n}$.

Consider the partition function of the lattice below
\begin{equation}
\iftikz
\label{cauchyrhs1}
\mathcal{P}(\xx_n, \yy_m) =
\begin{tikzpicture}[scale=0.6,baseline=(current bounding box.center)]
\foreach\x in {0,1,2}{
\bbull{2}{\x}{0.1};
\bbull{2}{3+\x}{0.1};
}
\node at (5,6.5) {$\xrightarrow[\hspace{40pt}]{\displaystyle \eta}$};
\node at (5,-1.5) {$\xrightarrow[\displaystyle \mu]{\hspace{40pt}}$};
\foreach\y in {2,3,4,5,6}{
\draw (\y+1,-1) -- (\y+1,6);}
\draw[ddscol,rounded corners](2,5) --  (8,5) -- (11,2);
\draw[ddscol,rounded corners](2,4) --  (8,4) -- (10.5,1.5);
\draw[ddscol,rounded corners](2,3) --  (8,3) -- (10,1);
\draw[scol,rounded corners](2,2) --  (8,2) -- (10,4);
\draw[scol,rounded corners](2,1) --  (8,1) -- (10.5,3.5);
\draw[scol,rounded corners](2,0) -- (8,0) -- (11,3);
\bbull{10}{1}{0.1};
\bbull{10.5}{1.5}{0.1};
\bbull{11}{2}{0.1};
\ebull{10}{4}{0.1};
\ebull{10.5}{3.5}{0.1};
\ebull{11}{3}{0.1};
\node at (7,6.2) {$ \cdots$};
\node at (7,-1.2) {$ \cdots$};
\node at (1.5,5) {$\ss y_1$};
\node at (1.5,4) {$\ss \vdots$};
\node at (1.5,3) {$\ss y_m$};
\node at (1.5,2) {$\ss x_n$};
\node at (1.5,1) {$\ss \vdots$};
\node at (1.5,0) {$\ss x_1$};
\end{tikzpicture}
\fi
\end{equation}
As the right edges are placed at infinity, the product of the $\mathfrak{R}$ factorizes
\begin{equation}
\iftikz
\label{cauchyrhs2}
\mathcal{P}(\xx_n, \yy_m) =
\begin{tikzpicture}[scale=0.6,baseline=(current bounding box.center)]
\foreach\x in {0,1,2}{
\draw[scol] (1,\x) -- (7,\x);
\draw[ddscol] (1,3+\x) -- (7,3+\x);
\bbull{1}{\x}{0.1};
\bbull{1}{3+\x}{0.1};
\bbull{7}{3+\x}{0.1};
\ebull{7}{\x}{0.1};
}
\node at (4,6.5) {$\xrightarrow[\hspace{40pt}]{\displaystyle \eta}$};
\node at (4,-1.5) {$\xrightarrow[\displaystyle \mu]{\hspace{40pt}}$};
\foreach\y in {1,2,3,4,5}{
\draw (\y+1,-1) -- (\y+1,6);}
\draw[ddscol] (9.5,5) -- (12.5,2);
\draw[ddscol] (9.5,4) -- (12,1.5);
\draw[ddscol] (9.5,3) -- (11.5,1);
\draw[scol] (9.5,2) -- (11.5,4);
\draw[scol] (9.5,1) -- (12,3.5);
\draw[scol] (9.5,0) -- (12.5,3);
\node at (6,6.2) {$ \cdots$};
\node at (6,-1.2) {$ \cdots$};
\foreach\x in {1,2,3}{
\bbull{9.5}{2+\x}{0.1};
\ebull{9.5}{\x-1}{0.1};}
\bbull{12.5}{2}{0.1};
\bbull{12}{1.5}{0.1};
\bbull{11.5}{1}{0.1};
\ebull{11.5}{4}{0.1};
\ebull{12}{3.5}{0.1};
\ebull{12.5}{3}{0.1};
\node at (8.2,3) {$\times$};
\node at (0.5,5) {$\ss y_1$};
\node at (0.5,4) {$\ss \vdots$};
\node at (0.5,3) {$\ss y_m$};
\node at (0.5,2) {$\ss x_n$};
\node at (0.5,1) {$\ss \vdots$};
\node at (0.5,0) {$\ss x_1$};
\end{tikzpicture}
\fi
\end{equation}
and so we have
\[
\dfrac{\mathcal{P}(\xx_n,\yy_m)}{\displaystyle \prod_{\substack{1\leq i\leq n\\[0.2em]1\leq j\leq m}}{(1-x_i y_j)}} = \sum_{\lambda} s_{\lambda/\mu}(\xx_n|{-\aaa}) \eschur^{\lambda/\eta}(\yy_m|\aaa^{n})
\]
We repeatedly apply the $\mathrm{RLL}$ relation to obtain an alternative expression for $\mathcal{P}(\xx_n,\yy_m)$.
\begin{equation}
\label{eq:RT_model}
\iftikz
\begin{tikzpicture}[scale=0.6,baseline=(current bounding box.center)]
\node at (-6,3) {$\mathcal{P}(\xx_n,\yy_m) =$};
\foreach\x in {0,1,2}{
\draw[ddscol] (2,\x) -- (8,\x);
\draw[scol] (2,3+\x) -- (8,3+\x);
\bbull{2}{\x}{0.1};
\bbull{2}{3+\x}{0.1};
\ebull{8}{3+\x}{0.1};
\bbull{8}{\x}{0.1};
}
\node at (5,6.5) {$\xrightarrow[\hspace{40pt}]{\displaystyle \eta}$};
\node at (5,-1.5) {$\xrightarrow[\displaystyle \mu]{\hspace{40pt}}$};
\foreach\y in {2,3,4,5,6}{
\draw (\y+1,-1) -- (\y+1,6);}
\node at (7,6.2) {$ \cdots$};
\node at (7,-1.2) {$ \cdots$};
\draw[scol] (2,5) -- (-1,2);
\bbull{-1}{2}{0.1};
\node at (-1.2,1.8) {$\ss x_n$};
\draw[scol] (2,4) -- (-0.5,1.5);
\node at (-0.7,1.5) {$\ss \ddots$};
\bbull{-0.5}{1.5}{0.1};
\draw[scol] (2,3) -- (0,1);
\node at (-0.2,0.7) {$\ss x_1$};
\bbull{0}{1}{0.1};
\draw[ddscol] (2,2) -- (0,4);
\node at (-0.3,4.2) {$\ss y_1$};
\bbull{0}{4}{0.1};
\bbull{-0.5}{3.5}{0.1};
\bbull{-1}{3}{0.1};
\draw[ddscol] (2,1) -- (-0.5,3.5);
\node at (-0.8,3.8) { \reflectbox{$\ss \ddots$}};
\draw[ddscol] (2,0) -- (-1,3);
\node at (-1.4,3.2) {$\ss y_m$};
\end{tikzpicture}
\fi
\end{equation}
In the above lattice, the configuration from the $R$ matrices is trivial. Hence, we get that 
\[
\mathcal{P}(\xx_n,\yy_m)= \sum_{\kappa} s_{\eta/\kappa}(\xx_n | {-\aaa}) E^{\mu/\kappa}(\yy_m|\aaa^n).
\]
Then the desired result is obtained by equating the two expressions of $\mathcal{P}(\xx_n,\yy_m)$.

We also remark that this equation is fixed under considering $\lambda$ (or $\kappa$) as a sequence of length $\ell \geq n$ with trailing zeros as we simply add particles that travel vertically to the model in~\eqref{eq:RT_model}, which have a Boltzmann weight of $1$.
\end{proof}
 
The function $\eschur^{\lambda}(\yy_m|\aaa^n)$ is not well-defined in the limit $n \to \infty$, so we cannot conclude any relationship with the dual Schur functions defined by Molev~\cite{Molev-dualschur} from the Cauchy--duality statement~\cite[Prop.~5.2]{Molev-dualschur}.

\subsection{Basis properties and variations}
\label{sec:edge_variations}

Next we discuss the properties of the basis of edge Schur functions and discuss some variations by multiplying by factor or specializing $\aaa = \alpha$.

We begin by remarking that the basis $\{\eschur^{\lambda}(\yy|\aaa)\}_{\lambda}$ does not seem to be a nice basis as it does not contain $1$ as $\eschur^{\emptyset}(\yy|\aaa)$ is an infinite sum and must lie in the completion.
Furthermore, the expansion of $\eschur^{\lambda}(\yy|\aaa) \eschur^{\mu}(\yy|\aaa)$ is not a finite sum; consider $\eschur^{1}(\yy_1|\aaa) \eschur^{1}(\yy_1|\aaa)$.
To account for this, one natural variation $\overline{\eschur}^{\lambda}$ would to disallow edge labels on the top boundary, or any horizontal edges for columns to the right of $\lambda_1$.
Therefore, we have
\begin{equation}
\label{eq:es_variant}
\overline{\eschur}^{\lambda}(\yy_m|\aaa) = \prod_{\substack{1 \leq j \leq m \\[0.2em] k > \lambda_1+1}} (1 + a_k y_i)^{-1} \eschur^{\lambda}(\yy_m|\aaa),
\end{equation}
and $\overline{\eschur}^{\emptyset}(\yy|\aaa) = 1$ and $\overline{\eschur}^{\lambda}(\yy_n|\aaa)$ to contain only a finite number of terms.
One drawback is these are not well-defined using Equation~\eqref{eq:es_variant} for skew shapes as we would have $E^{\emptyset}(\yy_n|\aaa) \neq E^{\lambda/\lambda}(\yy_n|\aaa)$ for any $\lambda \neq \emptyset$.
Instead, we would have to define a skew shape by the branching rule (\textit{cf.}~Equation~\eqref{eq:branching_rule})
\[
\overline{\eschur}^{\lambda}(\xx, \yy|\aaa) = \sum_{\mu \subseteq \lambda} \overline{\eschur}^{\lambda/\mu}(\yy|\aaa) \overline{\eschur}^{\mu}(\xx|\aaa),
\]
where we are only allowed to have edge labels on actual boxes.

Another variation on this would be to simply disallow all edge labels above the main diagonal, so we have the analog of~\eqref{eq:es_variant}:
\[
\dualfact^{\lambda/\mu}(\yy_m|\aaa) = \prod_{\substack{1 \leq j \leq m \\[0.2em] k \geq 0}} (1 + a_k y_i)^{-1} \eschur^{\lambda/\mu}(\yy_m|\aaa).
\]
Hence, this equals $\eschur^{\lambda/\mu}(\yy_n|\aaa_<)$, where $\aaa_< = (a_{-2}, a_{-1}, 0, 0, \ldots)$ denotes the specialization of $a_i = 0$ for all $i \geq 0$ and is the refined version of the $\CC^{\times}$-equivariant homology $H_{\bullet}^{\CC^{\times}}(\Gr)$ from the introduction.
These have the same properties as $\overline{\eschur}^{\lambda}(\yy|\aaa)$, namely $\dualfact^{\emptyset}(\yy|\aaa) = 1$ and $\dualfact^{\lambda/\mu}(\yy_n|\aaa) \in \ZZ[\yy_n,\aaa]$.
One advantage over $\overline{\eschur}^{\lambda}(\yy|\aaa)$ is they satisfy the same branching rule as the edge Schur functions~\eqref{eq:branching_rule} with no ambiguity for skew shapes since the ratio with $\eschur^{\lambda/\mu}(\yy_n|\aaa)$ does not depend on $\lambda$ or $\mu$.

A natural question would be to explicitly determine the functions that are dual to the factorial Schur functions under the Hall inner product.
Our skew Cauchy identity when $\mu = \eta = \emptyset$ becomes
\[
\sum_{\lambda} s_{\lambda}(\xx_n | {-\aaa}) \prod_{\substack{1\leq k \leq N \\[0.2em]1\leq j\leq m}}(1 + a_{k}y_{j})^{-1} \eschur^{\lambda}(\yy_m | \aaa^{n}) = \prod_{\substack{1\leq i\leq n\\[0.2em]1\leq j\leq m}}\frac{1}{1-x_i y_j},
\]
where the product on the left-hand side is $E^{\emptyset}(\yy_m|\aaa^n)$.
Therefore, we need to distribute all of the leading factors from the Cauchy identity:
\[
\mathscr{E}^{\lambda/\mu}(\yy_m|\aaa) = \prod_{\substack{-n \leq k\leq N \\[0.2em] 1\leq j \leq m}} (1+a_k y_j)^{-1} \eschur^{\lambda/\mu}(\yy_m|{-\aaa}),
\]
for some fixed constants $N \geq \lambda_1$ and $\ell(\lambda) \leq n$ with $a_i = 0$ for all $i > N$ (again, we allow $N = \infty$).
We can obtain a combinatorial description by distributing the leading factors across the vertex model, that is we form a modified $L$-matrix by $(1 - a_j x_i)^{-1} L_{ij}(x_i, a_j)$ (so each Boltzmann weight in~\eqref{eq:dual_schur_vertices} is multiplied by this factor; note $\mathsf{a}_1 = 1$).
This leads to a combinatorial description analogous to Proposition~\ref{prop:combinatorial_partition_function} using similar to edge labeled tableau but we instead allow the vertical edges to have finite multisets (possibly empty) with an analogous semistandard condition.
Scaling the $L$-matrices by a constant factor does not change the integrability or the $R$-matrix, and the same analysis performed above holds for $\mathscr{E}^{\lambda/\mu}(\yy_m|\aaa)$.

However, the $\mathscr{E}^{\lambda/\mu}$ are also less than satisfactory as they depend on the number of trailing zeros in $\lambda$.
Therefore, better functions to consider would be
\[
\widehat{\mathscr{E}}^{\lambda}(\yy_m|\aaa) = \prod_{\substack{-\ell(\lambda) \leq k\leq N \\[0.2em] 1\leq j \leq m}} (1+a_k y_j)^{-1} \eschur^{\lambda}(\yy_m|{-\aaa}) = \prod_{\substack{-n \leq k < -\ell(\lambda) \\[0.2em] 1\leq j \leq m}} (1+a_k y_j) \mathscr{E}^{\lambda}(\yy_m|\aaa).
\]
Unlike for $\mathscr{E}^{\emptyset}(\yy|\aaa)$, we have $\widehat{\mathscr{E}}^{\emptyset}(\yy|\aaa) = 1$.
Yet, they do not have a simple relationship for the skew functions similar to the $\overline{E}^{\lambda}$ variation.
Likewise, $\widehat{\mathscr{E}}^{\lambda/\mu}(\yy_m|\aaa)$ admits a simple combinatorial description by only allowing (vertical) edge labels on actual boxes.

Next, we define the (skew) \defn{dual Schur functions} $\widehat{s}_{\lambda/\mu}(\yy|\aaa)$ following Molev~\cite{Molev-dualschur} after substituting $a_i \leftrightarrow a_{-i}$, which are built up inductively using the branching rule
\[
\widehat{s}_{\lambda/\mu}(\yy_1|\aaa) = \prod_{\bbb \in \lambda/\mu} \frac{y_1}{1 - a_{c(\bbb)} y_1},
\qquad
\widehat{s}_{\lambda/\mu}(y', \yy|\aaa) = \sum_{\nu} \widehat{s}_{\nu/\mu}(\yy|\aaa) \prod_{\bbb \in \nu/\mu} \frac{1-a_{c(\bbb)-1} y'}{1-a_{c(\bbb)}y'} \widehat{s}_{\lambda/\nu}(y'|\aaa),
\]
where the sum is over all $\mu \subseteq \nu \subseteq \lambda$ such that $\lambda/\nu$ is a horizontal strip.
From this description, we can compute $\widehat{s}_{\lambda/\nu}(\yy_1|\aaa)$ as the partition function of the vertex model
\begin{equation}
\label{eq:substitution_model}
\begin{array}{c*{4}{@{\hspace{40pt}}c}}
\textsf{a}_1 & \textsf{b}_1 & \textsf{b}_2 & \textsf{c}_1 & \textsf{c}_2
\\
\begin{tikzpicture}[scale=0.4,baseline=-2pt]
\draw[arrow=0.25,mscol] (-1,0) node[left,black] {$\ss {0}$} -- (1,0) node[right,black] {$\ss {0}$};
\draw[arrow=0.25] (0,-1) node[below] {$\ss {0}$} -- (0,1) node[above] {$\ss{0}$};
\end{tikzpicture}
&
\begin{tikzpicture}[scale=0.4,baseline=-2pt]
\draw[arrow=0.25,mscol] (-1,0) node[left,black] {$\ss {0}$} -- (1,0) node[right,black] {$\ss {0}$};
\draw[arrow=0.25] (0,-1) node[below] {$\ss {1}$} -- (0,1) node[above] {$\ss {1}$};
\end{tikzpicture}
&
\begin{tikzpicture}[scale=0.4,baseline=-2pt]
\draw[arrow=0.25,mscol] (-1,0) node[left,black] {$\ss {1}$} -- (1,0) node[right,black] {$\ss {1}$};
\draw[arrow=0.25] (0,-1) node[below] {$\ss {0}$} -- (0,1) node[above] {$\ss {0}$};
\end{tikzpicture}
&
\begin{tikzpicture}[scale=0.4,baseline=-2pt]
\draw[arrow=0.25,mscol] (-1,0) node[left,black] {$\ss {0}$} -- (1,0) node[right,black] {$\ss {1}$};
\draw[arrow=0.25] (0,-1) node[below] {$\ss {1}$} -- (0,1) node[above] {$\ss {0}$};
\end{tikzpicture}
&
\begin{tikzpicture}[scale=0.4,baseline=-2pt]
\draw[arrow=0.25,mscol] (-1,0) node[left,black] {$\ss {1}$} -- (1,0) node[right,black] {$\ss {0}$};
\draw[arrow=0.25] (0,-1) node[below] {$\ss {0}$} -- (0,1) node[above] {$\ss {1}$};
\end{tikzpicture}
\\
1 & 1 & {\displaystyle \frac{y_i}{1-a_k y_1}} & 1 & {\displaystyle \frac{y_i}{1-a_k y_1}}
\end{array}
\end{equation}
similar to $\eschur^{\lambda/\nu}(\yy_m|\aaa)$ and $\mathscr{E}^{\lambda/\nu}(\yy_m|\aaa)$.
We thank Paul Zinn-Justin for noting this vertex model.
In particular, these differ from the Boltzmann weights to construct $\mathscr{E}^{\lambda/\nu}(\yy_m|\aaa)$ by multiplying $\textsf{b}_1$ and $\textsf{c}_1$ by $(1 - a_k y_1)$, and both of these vertices correspond to particles at the positions of $\nu$.
We simplify
\[
\prod_{\bbb \in \nu/\mu} \frac{1-a_{c(\bbb)-1} y'}{1-a_{c(\bbb)}y'} = \prod_{k=1}^n \frac{1-a_{\mu_k-k} y'}{1-a_{\nu_k-k}y'}
\]
by fixing the row $k$ and canceling common factors.
By noticing that $\nu_k-k$ are the positions of the particles, arrive at the same vertex model as for $\mathscr{E}^{\lambda/\nu}$ and have
\[
\prod_{k=1}^n (1-a_{\nu_k-k}y')^{-1} \widehat{s}_{\lambda/\nu}(y'|\aaa) = \mathscr{E}^{\lambda/\nu}(y'|\aaa).
\]
Hence, for $n \geq \ell(\lambda)$, from the branching rule we obtain the following relation.

\begin{cor}
\label{cor:dual_schur_comparison}
We have
\[
\widehat{s}_{\lambda/\nu}(\yy_m|\aaa) = \prod_{\substack{1 \leq k \leq \ell(\nu) \\[0.2em] 1 \leq j \leq m}} (1 - a_{\nu_k-k} y_j) \widehat{\mathscr{E}}^{\lambda/\nu}(\yy_m|\aaa).
\]
In particular, we have $\widehat{s}_{\lambda}(\yy_m|\aaa) = \widehat{\mathscr{E}}^{\lambda}(\yy_m|\aaa)$.
\end{cor}

Consequently, all of these variations are related by a product involving factors $(1 - a_k y_j)$.

\begin{remark}
The above definition for $\widehat{s}_{\lambda/\mu}(\xx|\aaa)$ is not the only natural way to define a skew version of the dual Schur functions.
There are alternatives that have better branching properties, such as being adjoint to multiplication of the factorial Schur functions.
It would be interesting to compare such a formulation to our skew edge Schur functions.
\end{remark}

For all of these variations, we note that they are all bases since they all are of the form $s_{\lambda} + HOT$.
It would be interesting to see what other properties these different bases have, such as positive (Schur) expansions or products decomposing into a finite sum.

We note that our vertex model satisfies the \defn{free fermion condition}, where the weights $\mathsf{a}_1 \mathsf{a}_2 + \mathsf{b}_1 \mathsf{b}_2 = \mathsf{c}_1 \mathsf{c}_2$, where $\mathsf{a}_2$ is (the weight of) the vertex with all boundaries $1$.
This does not naturally fit into the formalism defined by Hardt~\cite{Hardt21} since we have the column dependency.
However, if we set $\aaa = \alpha$ with $a_i = 0$ for $i \gg 1$, then, in the notation of~\cite[Fig.~2]{Hardt21} (written here in sans serif font), we can write
\begin{equation}
\label{eq:hardt_values}
\textsf{A}_i^{-1} = 1 + \alpha y_i,
\qquad\qquad
\textsf{B}_i^{-1} = A_i,
\qquad\qquad
\textsf{y}_i = y_i A_i,
\qquad\qquad
\textsf{x}_i = 0,
\end{equation}
after interchanging the Boltzmann weights of $\mathsf{c}_1 \leftrightarrow \mathsf{c}_2$.
(This change of Boltzmann weights does not affect the partition function; it only trivially changes how we identify boxes with the states.)
Then we can use~\cite[Cor.~5.2]{Hardt21} to obtain a Jacobi--Trudi type formula and~\cite[Prop.~5.7]{Hardt21} for a Pieri rule for this unrefined version.
This also allows us to express
\[
\eschur^{\lambda/\mu}(\yy_m|\alpha) = \prod_{j=1}^m (1 + \alpha y_j)^{N-n} s_{\lambda}[\yy_m / (1 + \alpha \yy_m)],
\]
where $n$ is the number of parts of $\lambda$ including trailing zeros, $N$ is the number of columns in the grid (which is necessarily at least $n$), and $s_{\lambda}[\yy_m / (1 + \alpha \yy_m)]$ denotes the substitution $y_j \mapsto y_j / (1 + \alpha y_j)$.
We note that this does not make sense for an infinite grid.
For the general case of $\aaa$, we expect the same results to hold for a refined version of the current operators defined as
\[
J_k(x) = \sum_{i \in \ZZ} s_i(x) s_{i-1}(x) \cdots s_{i-k+1}(x) \normord{\psi_{i-k}^* \psi_i}
\]
with $s_i(x) = a_i x$.
Note that $J_k(1)$ is the usual current operators and specializing $\aaa = \alpha$, we have $J_k(x) = (\alpha x)^k J_k(1)$ used to construct the (half) vertex operators.

Let us briefly discuss the functions $s_{\lambda}[\yy / (1 - \alpha \yy)]$, which was shown by Yeliussizov to be equal to the specialization $G_{\lambda}(\yy; \alpha,-\alpha)$ of a canonical Grothendieck polynomial~\cite[Cor.~3.5]{Yel17}.
We can see the vertex model~\eqref{eq:substitution_model} is the same substitution of the vertex model for Schur functions.
We note that under this specialization, the dual Schur functions satisfy the branching rule~\eqref{eq:branching_rule}.
Therefore, we have the following.
As far as the authors are aware has not appeared previously in the literature, but it can be easily seen directly from the combinatorial definition and is likely known to experts.

\begin{cor}
We have
\[
\widehat{s}_{\lambda}(\yy_m|\alpha) = s_{\lambda}[\yy_m / (1 - \alpha \yy_m)].
\]
\end{cor}

We can also introduce an extra parameter from~\cite{Hardt21} to get a six-vertex model.
If we take the extra parameter to be $\mathsf{x}_i = t y_j$, which reduces to our edge labeled tableaux case when $t = 0$, then we expect to obtain a generalization of Tokuyama's theorem for spherical Whittaker functions (see, \textit{e.g.},~\cite{BS18}).
Furthermore, we expect this to hold for the refined version (\textit{i.e.}, not specializing $\aaa$) and satisfy a Cauchy-type identity with the functions in~\cite{BMN14}.

One additional combinatorial question would be to consider what happens when we also allow edges to appear on the vertical edges that satisfies the natural semistandard conditions.
If we do not allow an $i$ to lie between two boxes with an $i$, then we can get a simple interpretation by modifying our lattice model.
This condition is equivalent that we instead have a strict inequality for the entries to the right of a set.
We can write this as a product $\prod_{i,j} (1 + a_j y_i) s_{\lambda}[\yy/(1+\aaa\yy)]$, where the weight of each box $\bbb \in T$ is instead $y_{\bbb} / (1 + a_{c(\bbb)} y_{\bbb})$.
We note that this is just a slightly modified version of our $L$-matrix~\eqref{eq:dual_schur_vertices} with $\mathsf{a}_1 = \mathsf{b}_1 = \mathsf{c}_1 = 1 + a_j y_i$.
There are also a variation that only permits the labels to appear on actual edges by multiplying by appropriate factors of $(1 + a_j y_i)^{-1}$.
This is the same as~\eqref{eq:hardt_values} but with $\mathsf{B}_i^{-1} = 1$ as this modified $L$-matrix still satisfies the free fermion condition.
As a consequence, we have a Jacobi--Trudi formula, Pieri rule, Cauchy identity, etc.\ when $\aaa = \alpha$.
We expect to be able to obtain a refined version of this Cauchy identity using a similar proof to Theorem~\ref{thm:cauchy}.

Back-stable Schubert calculus was recently extended to $K$-theory in~\cite{LLS21II}.
It would be interesting to see how our approach extends to the $K$-theory setting.


\section{Crystallization}
\label{sec:crystal_structure}

We now give a crystal structure on the set of edge labeled tableaux of shape $\lambda/\mu$.
We use this to give a Schur function expansion of $\eschur^{\lambda}(\xx | \aaa)$, which gives an alternative proof that they are symmetric functions (in an appropriate completion since they are unbounded in degree as long as $\aaa$ is infinite).
We also derive an alternative combinatorial formula using an analog of uncrowding for set-valued tableaux~\cite[Sec.~6]{Buch02}.

\subsection{Crystal preliminaries}

Let $\fsl_n$ denote the special linear Lie algebra of traceless $n \times n$ matrices (over $\CC$) with indexing set $I = [n-1]$ and coweight lattice $P^{\vee} = \ZZ^n = \operatorname{span}_{\ZZ}\{\epsilon^{\vee}_1, \dotsc, \epsilon^{\vee}_n\}$ with simple coroots $\{\coroot_i = \epsilon^{\vee}_i - \epsilon^{\vee}_{i+1} \}_{i \in I}$.
We can take the dual space $P$ with dual basis $\{\epsilon_1, \dotsc, \epsilon_n\}$, so $\epsilon^{\vee}_i(\epsilon_j) = \delta_{ij}$.
Define the simple roots $\{\alpha_i = \epsilon_i - \epsilon_{i+1}\}_{i \in I}$, fundamental weights $\{\fw_i = \epsilon_1 + \cdots + \epsilon_i \}_{i \in I}$, and an inner product $\inner{\alpha_i}{\coroot_j} = \alpha_i(\coroot_j) = C_{ij}$.
This matrix is the Cartan matrix
\[
[C_{ij}]_{i,j \in I} = \begin{bmatrix}
2 & -1 & 0 & \cdots & 0 \\
-1 & 2 & -1 & \ddots & 0 \\
\ddots & \ddots & \ddots & \ddots & \ddots \\
0 & \ddots & -1 & 2 & -1 \\
0 & \cdots & 0 & -1 & 2
\end{bmatrix}.
\]
We use the standard identification of partitions of length at most $n$ with elements in the dominant weight lattice $P^+ = \{ \mu \in P \mid \inner{\mu}{\coroot_i} \geq 0 \text{ for all } i \in I \}$ by $\lambda \leftrightarrow \sum_{i=1}^n \lambda_i \epsilon_i$.

An \defn{crystal} is a set $\mcB$ with \defn{crystal operators} $e_i, f_i \colon \mcB \to \mcB \sqcup \{ \zero \}$, for $i \in I$, such that for all $i \in I$, the functions
\[
\varepsilon_i(b) := \max \{k \mid e_i^k b \neq \zero\},
\qquad\qquad
\varphi_i(b) := \max \{k \mid f_i^k b \neq \zero\},
\qquad\qquad
\wt \colon B \to P,
\]
and all $b,b' \in \mcB$, the relations
\[
e_i b = b' \quad \Longleftrightarrow \quad b = f_i b',
\qquad\qquad
\inner{\wt(b)}{\coroot_i} + \varepsilon_i(b) = \varphi_i(b)
\]
hold and $\mcB$ corresponds to the crystal basis as defined by Kashiwara~\cite{K90,K91} of a Drinfel'd--Jimbo quantum group $U_q(\g)$ module.
Our definition is what is called a regular or seminormal crystal in the literature.
We refer the reader to~\cite{BS17} for additional information on crystals.
We call an element $b \in \mcB$ \defn{highest weight} if $e_i b = \zero$ for all $i \in I$.

For any $\lambda \in P^+$, there exists a unique crystal $B(\lambda)$ with a unique highest weight element $u_{\lambda}$ of weight $\lambda$ corresponding to the highest weight irreducible representation $V(\lambda)$~\cite{K90,K91}.
The \defn{character} of a crystal $\mcB$ is defined as
\[
\ch \mcB = \sum_{b \in \mcB} \prod_{i=1}^n x_i^{\wt(b)(\epsilon^{\vee}_i)}.
\]
It is a classical fact that $\ch B(\lambda) = s_{\lambda}(\xx_n)$.

We can construct the tensor product of crystals $\mcB_1, \dotsc, \mcB_L$ as follows.
Let $\mcB = \mcB_L \otimes \cdots \otimes \mcB_1$ be the set $\mcB_L \times \cdots \times \mcB_1$.
We define the crystal operators using the \defn{signature rule}.
Let $b = b_L \otimes \cdots \otimes b_2 \otimes b_1 \in \mcB$, and for $i \in I$, we write
\[
\underbrace{\cm\cdots\cm}_{\varphi_i(b_L)}\
\underbrace{\cp\cdots\cp}_{\varepsilon_i(b_L)}\
\cdots\
\underbrace{\cm\cdots\cm}_{\varphi_i(b_1)}\
\underbrace{\cp\cdots\cp}_{\varepsilon_i(b_1)}\ .
\]
Then by successively deleting any $(\cp\cm)$-pairs (in that order) in the above sequence, we obtain a sequence
\[
\sig_i(b) :=
\underbrace{\cm\cdots\cm}_{\varphi_i(b)}\
\underbrace{\cp\cdots\cp}_{\varepsilon_i(b)}
\]
called the \defn{reduced signature}.
Suppose $1 \leq j_{\cm}, j_{\cp} \leq L$ are such that $b_{j_{\cm}}$ contributes the rightmost $\cm$ in $\sig_i(b)$ and $b_{j_{\cp}}$ contributes the leftmost $\cp$ in $\sig_i(b)$.
Then, we have
\begin{align*}
e_i b &= b_L \otimes \cdots \otimes b_{j_{\cp}+1} \otimes e_ib_{j_{\cp}} \otimes b_{j_{\cp}-1} \otimes \cdots \otimes b_1, \\
f_i b &= b_L \otimes \cdots \otimes b_{j_{\cm}+1} \otimes f_ib_{j_{\cm}} \otimes b_{j_{\cm}-1} \otimes \cdots \otimes b_1.
\end{align*}
If one of the factors in a tensor product is $\zero$, then we consider the entire element to be $\zero$.
For type A, the highest weight condition is the classical Yamanouchi condition (see, \textit{e.g.},~\cite{ECII}).

\begin{remark}
Our tensor product convention follows~\cite{BS17}, which is opposite of the tensor product rule used by Kashiwara~\cite{K90,K91} (and that of~\cite{HK02}).
\end{remark}

For two crystals $\mcB_1$ and $\mcB_2$, a \defn{crystal morphism} $\psi \colon \mcB_1 \to \mcB_2$ is a map $\mcB_1 \sqcup \{\zero\} \to \mcB_2 \sqcup \{\zero\}$ with $\psi(\zero) = \zero$ such that the following properties hold for all $b \in B_1$ and $i \in I$:
\begin{itemize}
\item[(1)] If $\psi(b) \in \mcB_2$, then $\wt\bigl(\psi(b)\bigr) = \wt(b)$, $\varepsilon_i\bigl(\psi(b)\bigr) = \varepsilon_i(b)$, and $\varphi_i\bigl(\psi(b)\bigr) = \varphi_i(b)$.
\item[(2)] We have $\psi(e_i b) = e_i \psi(b)$ if $\psi(e_i b) \neq \zero$ and $e_i \psi(b) \neq \zero$.
\item[(3)] We have $\psi(f_i b) = f_i \psi(b)$ if $\psi(f_i b) \neq \zero$ and $f_i \psi(b) \neq \zero$.
\end{itemize}
An \defn{embedding} (resp.~\defn{isomorphism}) is a crystal morphism such that the induced map $B_1 \sqcup \{\zero\} \to B_2 \sqcup \{\zero\}$ is an embedding (resp.~bijection).

\begin{figure}
\[
\ytableausetup{boxsize=1.5em}
\begin{tikzpicture}[xscale=1.9,baseline=-4]
\node (1) at (0,0) {$\ytableaushort{1}$};
\node (2) at (1.5,0) {$\ytableaushort{2}$};
\node (3) at (3,0) {$\ytableaushort{3}$};
\node (d) at (4.5,0) {$\cdots$};
\node (n) at (6,0) {$\ytableaushort{n}$};
\draw[->,darkred] (1) to node[above]{\tiny$1$} (2);
\draw[->,blue] (2) to node[above]{\tiny$2$} (3);
\draw[->,dgreencolor] (3) to node[above]{\tiny$3$} (d);
\draw[->,UQpurple] (d) to node[above]{\tiny$n-1$} (n);
\end{tikzpicture}
\]
\caption{Crystal $B(\fw_1)$ of the natural representation of type $A_{n-1}$.}
\label{fig:vec_repr}
\end{figure}
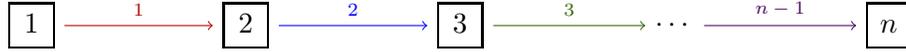

Next, we describe the crystal structure on semistandard tableaux of shape $\lambda / \mu$ (with entries in $[n]$), which we denote by $B(\lambda / \mu)$.
However, we will use a nonstandard reading word, but reading along diagonals from bottom-to-top.
Indeed, we read along a fixed diagonal from bottom-to-top, and we read along each diagonal from bottom-to-top.
This gives us an injection $\word \colon B(\lambda / \mu) \to B(\fw_1)^{\otimes\abs{\lambda / \mu}}$.

\begin{ex}
Under the reading word described above, we have
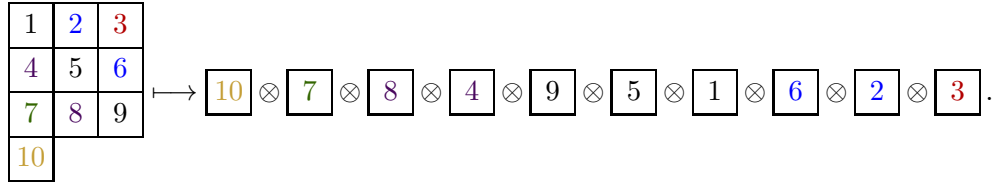
\[
\ytableausetup{boxsize=1.5em}
\ytableaushort{1{\color{blue}2}{\color{darkred}3},{\color{UQpurple}4}5{\color{blue}6},{\color{dgreencolor}7}{\color{UQpurple}8}9,{\color{UQgold}10}}
\longmapsto
\ytableaushort{{\color{UQgold}10}} \otimes \ytableaushort{{\color{dgreencolor}7}} \otimes \ytableaushort{{\color{UQpurple}8}} \otimes \ytableaushort{{\color{UQpurple}4}} \otimes \ytableaushort{9} \otimes \ytableaushort{5} \otimes \ytableaushort{1} \otimes \ytableaushort{{\color{blue}6}} \otimes \ytableaushort{{\color{blue}2}} \otimes \ytableaushort{{\color{darkred}3}}\,.
\]
If we group this by diagonals, we obtain the equivalent crystal
\[
\ytableaushort{{\color{UQgold}10}}
\otimes \ytableaushort{{\color{dgreencolor}7}}
\otimes \ytableaushort{{\color{UQpurple}4}{\color{UQpurple}8}}
\otimes \ytableaushort{159}
\otimes \ytableaushort{{\color{blue}2}{\color{blue}6}}
\otimes \ytableaushort{{\color{darkred}3}}\,.
\]
\end{ex}

We say a reading word is \defn{admissible} if for any fixed box $\bbb$, we read every box to its northeast after $\bbb$.
It is straightforward to see that $\word$ defines an admissible reading, and so we have the following by~\cite[Thm.~7.3.6]{HK02}.

\begin{thm}
The set $\ssyt(\lambda / \mu)$ is closed under the crystal operators induced from $\word$.
Moreover, this is the same crystal structure on $\ssyt(\lambda/\mu)$ as the one using the Far-Eastern reading word.
\end{thm}

We note that if we take the reverse Far-Eastern reading word, an semistandard tableau has the lattice property if and only if it is a highest weight element.
This is a direct translation of the signature rule.

\subsection{Crystal structure}

We define our crystal structure on edge labeled tableaux by extending the reading for a box $\bbb$ with an entry $b$ and a set $A = \{a_1 < \cdots < a_k\}$ on the edge below $\bbb$.
For each such box, we read this as
\[
\ytableausetup{boxsize=1.5em}
\ytableaushort{{\elb{b}{A}}}
\quad \Longleftrightarrow \quad
\ytableaushort{{a_k}} \otimes \cdots \otimes \ytableaushort{{a_1}} \otimes \ytableaushort{b}
\quad \Longleftrightarrow \quad
\ytableaushort{b,{a_1},{\raisebox{-2pt}{$\vdots$}},{a_k}}
\]
We read the tableau following the reading word $\word$ using this description for each box, including for a box ``above'' the top boundary.

\begin{ex}
Consider the edge labeled tableau $T$ from Example~\ref{ex:skew_ELT}.
Then the reading word is
\[
\newcommand{\bl}[1]{{\color{OCUenji} \textbf{#1}}}
\bl{5} \, 6\bl{5} 42\bl{1} \, 3\bl{2} \, \bl{5} 1 \, 63 \, 52\bl{1},
\]
where the bold letters are the entries in each box.

\[
\ytableausetup{boxsize=1.5em}
T = \ytableaushort{{\cdot}{\elb{\cdot}{1}}{\cdot}{\elb{\cdot}{36}}{\elb{1}{25}},{\elb{1}{24}}{\elb{2}{3}}5,5{\elb{5}{6}}}\,,
\]
\end{ex}

We define the crystal structure by using the signature rule with this reading word.
This generally gives a valid edge labeled tableau with the following exception:
\[
\ytableausetup{boxsize=1.5em}
\ytableaushort{i{\elb{i}{p}}}
\quad
\xrightarrow[\hspace{30pt}]{i}
\quad
\ytableaushort{{\elbb{i}{p}}{p}}\,,
\]
where $p = i+1$.
Indeed, we can characterize this as occurring for $f_i$ when we are acting on an entry $i$ in a box $\bbb$ such that the entry immediately to its right is also $i$.
For $e_i$, we are acting on an entry $i+1$ in a box $\bbb$ such that there the edge immediately above it contains an $i$, and we perform the inverse operation.

\begin{prop}
\label{prop:reading_word_commutes}
We have $\word(f_i T) = f_i \word(T)$.
\end{prop}

\begin{proof}
We note that the $i$ and $i+1$ form a canceling pair, and there is no entry in between these in the reading word.
Thus, we act on the $i$ from the left box, and change it into an $i+1$.
\end{proof}

\begin{thm}
\label{thm:edge_labeled_crystal}
This gives a crystal structure isomorphic to a direct sum of highest weight crystals.
\end{thm}

\begin{proof}
The reading word gives an explicit bijection between a diagonal and $B(\fw_k)$, where $k$ is the number of entries along the diagonal.
Note that the semistandard condition implies that every element along a diagonal is distinct.
From Proposition~\ref{prop:reading_word_commutes}, we have that the edge labeled tableaux are a subcrystal of
\[
B = \bigotimes_{i=1}^m B(\fw_{k_i}) \iso \bigoplus_{i=1}^N B(\lambda^{(i)})^{\oplus m_i}.
\]
The claim follows from the fact edge labeled tableaux are closed under the crystal operators.
\end{proof}

\begin{cor}
For any $\lambda$, we have
\[
\eschur^{\lambda}(\xx|\aaa) = \sum_{\lambda \subseteq \nu} c_{\lambda}^{\nu}(\aaa) s_{\mu}(\xx),
\]
where $c_{\lambda}^{\nu}(\aaa) = \sum_T \wt(T)\rvert_{x_i = 1}$ summed over all highest weight edge labeled tableaux of shape $\lambda$ and weight (as a crystal element) $\mu$.
\end{cor}

This gives an alternative proof of Theorem~\ref{thm:elambda-symmetric}.
Furthermore, from Corollary~\ref{cor:dual_schur_comparison} and noting that this product is equivalent to setting $a_i = 0$ for all $i > \lambda_1$, we see that $\widehat{s}_{\lambda}(\xx|{-\aaa})$ expands in the Schur functions with coefficients in $\ZZ_{\geq 0}[\aaa]$, which also follows from~\cite[Thm.~3.17]{Molev-dualschur}.

\begin{figure}
\iftikz
\begin{gather*}
\begin{tikzpicture}[xscale=3.4,yscale=1,>=latex]
\node (hw) at (0,0) {$\ytableaushort{111,2{\elb{2}{3}}}$};
\node (f1) at (1,1) {$\ytableaushort{112,2{\elb{2}{3}}}$};
\node (f21) at (2,1) {$\ytableaushort{113,2{\elb{2}{3}}}$};
\node (f221) at (3,1) {$\ytableaushort{{\elb{1}{2}}13,33}$};
\node (f2) at (1,-1) {$\ytableaushort{{\elb{1}{2}}11,33}$};
\node (f12) at (2,-1) {$\ytableaushort{{\elb{1}{2}}12,33}$};
\node (f112) at (3,-1) {$\ytableaushort{{\elb{1}{2}}22,33}$};
\node (lw) at (4,0) {$\ytableaushort{{\elb{1}{2}}23,33}$};
\draw[->,darkred] (hw) to node[above]{\tiny$1$} (f1);
\draw[->,darkred] (f2) to node[above]{\tiny$1$} (f12);
\draw[->,darkred] (f12) to node[above]{\tiny$1$} (f112);
\draw[->,darkred] (f221) to node[above]{\tiny$1$} (lw);
\draw[->,blue] (hw) to node[above]{\tiny$2$} (f2);
\draw[->,blue] (f1) to node[above]{\tiny$2$} (f21);
\draw[->,blue] (f21) to node[above]{\tiny$2$} (f221);
\draw[->,blue] (f112) to node[above]{\tiny$2$} (lw);
\end{tikzpicture}
\allowdisplaybreaks \\
\begin{tikzpicture}[xscale=3.4,yscale=1,>=latex]
\node (hw) at (0,0) {$\ytableaushort{111,2{\elbb{2}{3}}}$};
\node (f1) at (1,1) {$\ytableaushort{112,2{\elbb{2}{3}}}$};
\node (f21) at (2,1) {$\ytableaushort{113,2{\elbb{2}{3}}}$};
\node (f221) at (3,1) {$\ytableaushort{113,3{\elbb{2}{3}}}$};
\node (f2) at (1,-1) {$\ytableaushort{111,3{\elbb{2}{3}}}$};
\node (f12) at (2,-1) {$\ytableaushort{112,3{\elbb{2}{3}}}$};
\node (f112) at (3,-1) {$\ytableaushort{{\elbb{1}{2}}22,33}$};
\node (lw) at (4,0) {$\ytableaushort{{\elbb{1}{2}}23,33}$};
\draw[->,darkred] (hw) to node[above]{\tiny$1$} (f1);
\draw[->,darkred] (f2) to node[above]{\tiny$1$} (f12);
\draw[->,darkred] (f12) to node[above]{\tiny$1$} (f112);
\draw[->,darkred] (f221) to node[above]{\tiny$1$} (lw);
\draw[->,blue] (hw) to node[above]{\tiny$2$} (f2);
\draw[->,blue] (f1) to node[above]{\tiny$2$} (f21);
\draw[->,blue] (f21) to node[above]{\tiny$2$} (f221);
\draw[->,blue] (f112) to node[above]{\tiny$2$} (lw);
\end{tikzpicture}
\allowdisplaybreaks \\
\begin{tikzpicture}[xscale=3.4,yscale=1,>=latex]
\node (hw) at (0,0) {$\ytableaushort{11{\elb{1}{2}},23}$};
\node (f1) at (1,1) {$\ytableaushort{1{\elbb{1}{2}}2,23}$};
\node (f21) at (2,1) {$\ytableaushort{1{\elbb{1}{2}}3,23}$};
\node (f221) at (3,1) {$\ytableaushort{1{\elbb{1}{2}}3,33}$};
\node (f2) at (1,-1) {$\ytableaushort{11{\elb{1}{2}},33}$};
\node (f12) at (2,-1) {$\ytableaushort{1{\elbb{1}{2}}2,33}$};
\node (f112) at (3,-1) {$\ytableaushort{2{\elbb{1}{2}}2,33}$};
\node (lw) at (4,0) {$\ytableaushort{2{\elbb{1}{2}}3,33}$};
\draw[->,darkred] (hw) to node[above]{\tiny$1$} (f1);
\draw[->,darkred] (f2) to node[above]{\tiny$1$} (f12);
\draw[->,darkred] (f12) to node[above]{\tiny$1$} (f112);
\draw[->,darkred] (f221) to node[above]{\tiny$1$} (lw);
\draw[->,blue] (hw) to node[above]{\tiny$2$} (f2);
\draw[->,blue] (f1) to node[above]{\tiny$2$} (f21);
\draw[->,blue] (f21) to node[above]{\tiny$2$} (f221);
\draw[->,blue] (f112) to node[above]{\tiny$2$} (lw);
\end{tikzpicture}
\allowdisplaybreaks \\
\begin{tikzpicture}[xscale=2.47,yscale=1,>=latex,every node/.style={scale=0.8}]
\node (hw) at (0,-1) {$\ytableaushort{11{\elb{1}{2}},22}$};
\node (f2) at (1,0) {$\ytableaushort{11{\elb{1}{3}},22}$};
\node (f22) at (2,1) {$\ytableaushort{11{\elb{1}{3}},23}$};
\node (f222) at (3,2) {$\ytableaushort{11{\elb{1}{3}},33}$};
\node (f12) at (2,-1) {$\ytableaushort{11{\elb{2}{3}},22}$};
\node (f122) at (3,0) {$\ytableaushort{11{\elb{2}{3}},23}$};
\node (f1122) at (4,-1) {$\ytableaushort{12{\elb{2}{3}},33}$};
\node (f1222) at (4,1) {$\ytableaushort{11{\elb{2}{3}},33}$};
\node (f11222) at (5,0) {$\ytableaushort{12{\elb{2}{3}},33}$};
\node (f111222) at (6,-1) {$\ytableaushort{22{\elb{2}{3}},33}$};
\draw[->,darkred] (f2) to node[above,scale=1/.8]{\tiny$1$} (f12);
\draw[->,darkred] (f22) to node[above,scale=1/.8]{\tiny$1$} (f122);
\draw[->,darkred] (f122) to node[above,scale=1/.8]{\tiny$1$} (f1122);
\draw[->,darkred] (f222) to node[above,scale=1/.8]{\tiny$1$} (f1222);
\draw[->,darkred] (f1222) to node[above,scale=1/.8]{\tiny$1$} (f11222);
\draw[->,darkred] (f11222) to node[above,scale=1/.8]{\tiny$1$} (f111222);
\draw[->,blue] (hw) to node[above,scale=1/.8]{\tiny$2$} (f2);
\draw[->,blue] (f2) to node[above,scale=1/.8]{\tiny$2$} (f22);
\draw[->,blue] (f22) to node[above,scale=1/.8]{\tiny$2$} (f222);
\draw[->,blue] (f12) to node[above,scale=1/.8]{\tiny$2$} (f122);
\draw[->,blue] (f122) to node[above,scale=1/.8]{\tiny$2$} (f1222);
\draw[->,blue] (f1122) to node[above,scale=1/.8]{\tiny$2$} (f11222);
\end{tikzpicture}
\end{gather*}
\fi
\caption{Crystals isomorphic to $B(\fw_3 + \fw_2 + \fw_1)$ and $B(3\fw_2)$ for $\sl_3$.}
\label{fig:crystal_examples}
\end{figure}

\begin{ex}
Let $\lambda = (3, 2)$.
For any $n \geq 3$, we have
\[
\eschur^{32}(\xx_n|\aaa) = s_{32}(\xx_n) + (a_{-2} + a_{-1} + a_0 + a_1) s_{321}(\xx_n) + a_1 s_{33}(\xx_n) + \sum_{i > 1} a_i s_{42} + HOT.
\]
We have the crystals for the coefficients $a_{-1}$, $a_0$, and $a_1$ with $n = 3$ given by Figure~\ref{fig:crystal_examples}.
\end{ex}

\subsection{Uncrowding bijection}
\label{sec:uncrowding}

In this section, we construct an analog of the uncrowding bijection in analogy to~\cite{Buch02,HS20,LamPyl07}.
In this case, given our reading word, we will perform the uncrowding along diagonals, which requires a little more care.

\begin{dfn}[Uncrowding algorithm]
We proceed along diagonals starting from the lower-left box.
Start with $(P_0, Q_0) = (\emptyset, \emptyset)$.
For the $i$-th diagonal $D_i$, let $P_i = P_{i-1} \xleftarrow{RSK} \word(D_i)$ denote the RSK insertion.
We construct the recording tableau $Q_i$ as the skew shape $\mu^{(i)} / \nu^{(i)}$, where $\mu^{(i)}$ is the shape of $P_i$ and $\nu^{(i)}$ are the boxes of $\lambda$ up to the $i$-th diagonal (counted from the lower-left box) and slid up into a straight shape.
The entries of $Q_i$ are those of $Q_{i-1}$ shifted appropriately and then any remaining empty cells are filled with an $i$.
\end{dfn}

It would be good to describe this using a formulation analogous to the alternative descriptions for uncrowding given in~\cite{MPPS19,PPPS20}.

\begin{ex}
Consider the edge labeled tableau
\[
\ytableausetup{boxsize=1.5em}
T = \ytableaushort{11{\elb{2}{45}},226,3{\elb{3}{4}},{\elb{4}{5}}5}\,.
\]
Under the uncrowding algorithm, we have
\begin{alignat*}{3}
(P_0, Q_0) & = \emptyset, && \emptyset && \xleftarrow{RSK} 54
\\
(P_1, Q_1) & = \ytableaushort{4,5}\,, && \ytableaushort{{\cdot},1} && \xleftarrow{RSK} 53
\allowdisplaybreaks \\
(P_2, Q_2) & = \ytableaushort{35,4,5}\,, && \ytableaushort{{\cdot}{\cdot},{\cdot},1} && \xleftarrow{RSK} 432
\allowdisplaybreaks \\
(P_3, Q_3) & = \ytableaushort{23,34,45,5}\,, && \ytableaushort{{\cdot}{\cdot},{\cdot}{\cdot},{\cdot}3,1} && \xleftarrow{RSK} 21
\allowdisplaybreaks \\
(P_4, Q_4) & = \ytableaushort{12,23,34,45,5}\,, && \ytableaushort{{\cdot}{\cdot},{\cdot}{\cdot},{\cdot}{\cdot},{\cdot}3,1} && \xleftarrow{RSK} 61
\allowdisplaybreaks \\
(P_5, Q_5) & = \ytableaushort{116,22,33,44,55}\,, && \ytableaushort{{\cdot}{\cdot}{\cdot},{\cdot}{\cdot},{\cdot}{\cdot},{\cdot}{\cdot},13} \qquad && \xleftarrow{RSK} 542
\allowdisplaybreaks \\
(P_6, Q_6) & = \ytableaushort{1124,225,336,44,55}\,, \quad && \ytableaushort{{\cdot}{\cdot}{\cdot}{6},{\cdot}{\cdot}{\cdot},{\cdot}{\cdot}{6},{\cdot}{\cdot},13}
\end{alignat*}
\end{ex}

Now we need to describe the inverse algorithm; in particular, we need to describe which which cells to remove as we will use inverse RSK at each step.
We proceed by removing the diagonals in reverse order but starting with the cell at the bottom of the corresponding column.
We also remove any cell labeled by $i$ if we are at the $i$-th diagonal from the bottom, the result of which becomes an edge label and can be placed in a unique way such that the result is an edge labeled tableau.

\begin{thm}
\label{thm:uncrowdingmap}
The uncrowding map
\[
\Upsilon \colon \elt_{\lambda} \to \bigsqcup_{\mu \subseteq \lambda} \ssyt_{\mu} \times \mathcal{H}_{\lambda/\mu}
\]
is a crystal isomorphism for some set of tableau $\mathcal{H}_{\lambda/\mu}$.
\end{thm}

\begin{proof}
From the above, we have that the uncrowding operation is a bijection.
It commutes with the crystal operators since RSK is a crystal isomorphism (see, \textit{e.g.},~\cite[Thm.~8.7]{BS17}).
\end{proof}

This gives an alternative proof of Theorem~\ref{thm:edge_labeled_crystal}.
While we do not characterize the recording tableaux, we conjecture $\mathcal{H}_{\lambda/\mu}$ is essentially the set of hook $\lambda/\mu$-tableaux from~\cite[Thm.~3.17]{Molev-dualschur}.

\begin{conj}
\label{conj:recording_tableaux}
When restricted to the tableau for $\widehat{s}_{\lambda}(\xx|\aaa)$, the set $\mathcal{H}_{\lambda/\mu}$ are the hook $\lambda/\mu$-tableau (with an appropriate change of labels).
\end{conj}

Note that answering Conjecture~\ref{conj:recording_tableaux} would provide a new proof of~\cite[Thm.~3.17]{Molev-dualschur}.
If it is false, then there should be an uncrowding algorithm where the hook $\lambda/\mu$-tableau are the set of recording tableaux.

\begin{problem}
Determine if there exist crystal structures for the variations of the edge Schur functions.
If so, determine the corresponding uncrowding algorithm describing the Schur expansion.
\end{problem}

\begin{remark}
We could have alternatively read each box by using the edge above the box (instead of below).
We have the exact same exception to using the usual crystal structure (although the failure for the direct rule is different) and an isomorphic crystal.
In fact, these construct the same crystal structure on edge labeled tableaux as the reading words differ by Knuth relations, which yields the same insertion tableau after uncrowding (although their recording tableau may differ).
\end{remark}

\bibliographystyle{alpha}
\bibliography{grothendiecks}{}
\end{document}